\documentclass[11pt,reqno]{amsart}

\usepackage[T1]{fontenc}
\usepackage{amsmath, amssymb, amsthm, amsfonts}
\usepackage{mathtools}
\usepackage[dvipsnames]{xcolor}

\definecolor{DartmouthGreen}{HTML}{00693E}
\definecolor{DartmouthDarkGreen}{HTML}{004B23}
\definecolor{DartmouthLightGreen}{HTML}{4BAF7A}

\usepackage[
  colorlinks=true,
  hyperindex,
  linkcolor=DartmouthDarkGreen,
  citecolor=DartmouthGreen,
  urlcolor=DartmouthLightGreen,
  pagebackref=true
]{hyperref}

\usepackage{fullpage}
\usepackage{tikz,tikz-cd}
\usetikzlibrary{graphs, graphs.standard,positioning}

\usepackage{colortbl, array,booktabs,tabularx}
\usepackage{enumerate}
\usepackage[shortlabels]{enumitem}

\usepackage{calrsfs}
\DeclareMathAlphabet{\pazocal}{OMS}{zplm}{m}{n}
\usepackage{newcent,fouriernc}

\newtheorem{theorem}{Theorem}[section]
\newtheorem{lemma}[theorem]{Lemma}

\newtheorem{corollary}[theorem]{Corollary}
\newtheorem{proposition}[theorem]{Proposition}
\newtheorem{conjecture}[theorem]{Conjecture}

\theoremstyle{definition}
\newtheorem{definition}[theorem]{Definition}
\newtheorem{example}[theorem]{Example}

\newtheorem{remark}[theorem]{Remark}

\newtheorem{manualtheoreminner}{Theorem}
\newenvironment{manualtheorem}[1]{%
  \IfBlankTF{#1}
    {\renewcommand{\themanualtheoreminner}{\unskip}}
    {\renewcommand\themanualtheoreminner{#1}}%
  \manualtheoreminner
}{\endmanualtheoreminner}

\newtheorem{manualcorinner}{Corollary}
\newenvironment{manualcorollary}[1]{%
  \IfBlankTF{#1}
    {\renewcommand{\themanualcorinner}{\unskip}}
    {\renewcommand\themanualcorinner{#1}}%
  \manualcorinner
}{\endmanualcorinner}

\newtheorem{manualconjinner}{Conjecture}

\newenvironment{manualconjecture}[1]{%
  \IfBlankTF{#1}
    {\renewcommand{\themanualconjinner}{\unskip}}
    {\renewcommand{\themanualconjinner}{#1}}%
  \manualconjinner
}{\endmanualconjinner}

\newcommand{\bP}{\mathbf{P}}

\newcommand{\FF}{\mathbb{F}}

\newcommand{\QQ}{\mathbb{Q}}
\newcommand{\RR}{\mathbb{R}}

\newcommand{\ZZ}{\mathbb{Z}}

\newcommand{\cB}{\mathcal{B}}

\newcommand{\cI}{\mathcal{I}}

\newcommand{\cM}{\mathcal{M}}

\newcommand{\cR}{\mathcal{R}}

\newcommand{\sM}{\mathsf{M}}

\newcommand{\sQ}{\mathsf{Q}}
\newcommand{\sP}{\mathsf{P}}

\newcommand{\GC}{\mathcal{GC}}
\newcommand{\MC}{\mathcal{M}}

\DeclareMathOperator{\Aut}{Aut}
\DeclareMathOperator{\id}{id}
\DeclareMathOperator{\Span}{Span}
\DeclareMathOperator{\Out}{Out}

\DeclareMathOperator{\sgn}{sgn}

\DeclareMathOperator{\img}{img}
\DeclareMathOperator{\Sym}{Sym}

\DeclareMathOperator{\rk}{rk}
\DeclareMathOperator{\nul}{nul}

\DeclareMathOperator{\del}{del}
\DeclareMathOperator{\con}{con}
\DeclareMathOperator{\tot}{tot}
\DeclareMathOperator{\Tot}{Tot}

\DeclareMathOperator{\lp}{lp}
\DeclareMathOperator{\clp}{clp}
\DeclareMathOperator{\cont}{cont}
\DeclareMathOperator{\disc}{disc}

\DeclareMathOperator{\reg}{\textbf{reg}}
\DeclareMathOperator{\lpl}{\textbf{lpl}}
\DeclareMathOperator{\colpl}{\textbf{clpl}}
\DeclareMathOperator{\simp}{\textbf{simp}}
\DeclareMathOperator{\cosimp}{\textbf{csimp}}
\DeclareMathOperator{\grph}{\textbf{grph}}
\DeclareMathOperator{\cogrph}{\textbf{cgrph}}
\DeclareMathOperator{\gr}{gr}
\DeclareMathOperator{\GL}{GL}

\makeatletter
\newcommand*{\tp}{%
  {\mathpalette\@transpose{}}%
}
\newcommand*{\@transpose}[2]{%
  \raisebox{\depth}{$\m@th#1\intercal$}%
}
\makeatother

\newcommand{\E}{\cellcolor{blue!30}}
 
\newcommand{\matroidtable}[1]{%
\begin{tikzpicture}
\node[inner sep=0pt] (tbl) {%
\renewcommand{\arraystretch}{1.5}%
\setlength{\tabcolsep}{6pt}%
\begin{tabular}{r | *{9}{>{\centering\arraybackslash}m{2.2em}} >{\centering\arraybackslash}m{5.2em}}
#1
\end{tabular}%
};
\node[left=10pt of tbl.west, rotate=90, anchor=center, font=\itshape] {$r$};
\node[below=6pt of tbl.south, anchor=center, font=\itshape] {$n$};
\end{tikzpicture}%
}
\title{Explorations of Matroid Complexes}

\author[J. Bruce]{Juliette Bruce}
\address{Department of Mathematics, Dartmouth College, Hanover, NH 03755}
\email{\href{mailto:juliette.bruce@dartmouth.edu}{juliette.bruce@dartmouth.edu}}
\urladdr{\url{https://www.juliettebruce.xyz}}

\author[J. Bucciarelli]{Jacob Bucciarelli}
\address{Department of Mathematics, Kansas State University, Manhattan, KS 66506}
\email{\href{mailto:jbucciarelli@ksu.edu}{jbucciarelli@ksu.edu}}
\urladdr{}

\author[B. Zacovic]{Bailee Zacovic}
\address{Department of Mathematics, University of Michigan, Ann Arbor, MI 48109}
\email{\href{mailto:bzacovic@umich.edu}{bzacovic@umich.edu}}
\urladdr{\url{https://bzacovic.github.io/Bailee-Zacovic/}}

\begin{document}

%%%%%%%%%%%%%%%%%%%%%%%%%%%%%%%%%%%%%%%%%%%%%%%%%%%
%%%%%%%%%%%%%%%%%%%%%%%%%%%%%%%%%%%%%%%%%%%%%%%%%%%

\begin{abstract}
Motivated by Kontsevich's graph complexes, this paper gives a systematic study of matroid complexes. We construct deletion and contraction bicomplexes on the vector space spanned by matroid classes equipped with ground-set orientations, organizing the several naturally arising variants into a single unified framework. We show that direct sum and restriction-contraction make this space into a connected graded Hopf algebra extending Schmitt's matroid Hopf algebra, and use the resulting dg-algebra structure to prove broad acyclicity results. We compute the total, simple, loopless, regular, binary, and ternary matroid complexes through ground-set size $9$, and the connected quotient of the simple loopless regular complex through ground-set size $15$. These computations detect nontrivial homology and lead to a conjectural description in terms of odd-wheel matroids.
\end{abstract}

\maketitle
\setcounter{tocdepth}{1}
\tableofcontents

%%%%%%%%%%%%%%%%%%%%%%%%%%%%%%%%%%%%%%%%%%%%%%%%%%%
\section{Introduction}
%%%%%%%%%%%%%%%%%%%%%%%%%%%%%%%%%%%%%%%%%%%%%%%%%%%

Kontsevich's graph complexes have become a forum for combinatorics, topology, geometry, and mathematical physics. Originally introduced as chain complexes spanned by graphs modulo natural symmetries \cite{kontsevich93,kontsevich94}, they have exhibited deep relations to a wide range of areas: the Grothendieck--Teichm\"uller Lie algebra \cite{willwacher15}, the cohomology of moduli spaces \cite{CGP21,brown21,CFGP23,payneWillwacher24}, the homology of $\Out(F_n)$ and the topology of Culler--Vogtmann's Outer space \cite{cullerVogtmann86,borinskyVogtmann20,borinskyVogtmann23}, the cohomology of arithmetic groups \cite{brownHuPanzer24,brown25}, and quantum field theory via Feynman integrals/diagrams \cite{willwacher18,borinskyVogtmann22,hbHu25,Z19}. At the same time, large-scale computations have proven essential \cite{DBNBM01,brunWillwacher24,willwacher25}.  It is therefore natural to ask whether analogous complexes built from other geometric-combinatorial objects display comparable structure.
  
Matroids have hovered at the edge of this story. Every graph determines a matroid via its cycle structure, and matroids have deletion and contraction operations analogous to those used to define graph complexes. Indeed, Alekseyevskaya, Borovik, Gelfand, and White \cite{ABGW00} introduced \emph{matroid complexes} as a direct generalization of Kontsevich's graph complexes already in the late 1990s, replacing graphs by matroids and using matroid deletion and contraction to define differentials. However, while graph complexes subsequently attracted intense study, matroid complexes received comparatively little attention. This gap, together with the appearance of a related chain complex of simple regular matroids in the study of the cohomology of the moduli space of principally polarized abelian varieties (equivalently, the cohomology of the symplectic group) \cite{BBCMMW24}, motivates the systematic study undertaken here.

We work with the $\QQ$-vector space $\MC_\bullet$ generated by matroid isomorphism classes equipped with ground-set orientations, graded by ground-set size. Using deletion and contraction, we construct four differentials $\partial_{\del}$, $\partial_{\clp}$, $\partial_{\con}$, $\partial_{\lp}$ on $\MC_\bullet$ (the deletion of non-coloops, the deletion of coloops, the contraction of non-loops, and the contraction of loops) and organize them into two bicomplexes: the deletion bicomplex $(\MC_{\bullet,\bullet},\partial_{\del},\partial_{\clp})$ and the contraction bicomplex $(\MC_{\bullet,\bullet},\partial_{\lp},\partial_{\con})$, related by matroid duality. These bicomplexes recover the variants appearing in \cite{ABGW00} and \cite{BBCMMW24} via totalization, associated graded pieces of the rank and nullity filtrations, and restriction to combinatorial subclasses such as simple, loopless, regular, binary, and ternary matroids. They also expose useful algebraic structure on $\MC_\bullet$: direct sum gives a product, restriction-contraction gives a coproduct, and the interaction of these operations with the differentials leads to strong formal consequences.
  
A theme that runs throughout much of the literature on graph complexes is that these complexes and their homology carry rich algebraic structure, e.g., Lie/Hopf structures, filtrations by loop and edge number, etc.  These structures have frequently been leveraged to great effect, playing an essential role in organizing computations, forcing structural results, and detecting nontrivial homology. For example, the zeroth homology of the (commutative) graph complex is isomorphic to the Grothendieck–Teichm\"{u}ller Lie algebra \cite{willwacher15}, and this in turn is a key input for showing the dimensions of certain cohomology groups of the moduli space of genus $g$ curves grow exponentially in $g$. With this motivation we prove that Schmitt's Hopf algebra of matroids \cite{schmitt94} can be extended to $\MC_{\bullet}$ in a way that interacts meaningfully with the differentials. 
 
\begin{manualtheorem}{\ref{thm:hopf-algebra}}
  The direct sum product $\star$ and the restriction-contraction coproduct $\Delta$ make $\MC_\bullet$ into a connected graded Hopf algebra, with unit $[\varnothing,1]$ and counit $\epsilon$, where $\epsilon([\varnothing,1])=1$ and is zero otherwise. Further, each of the differentials is a graded derivation of $\star$.
  \end{manualtheorem}

 The Hopf algebra is not, however, a dg-Hopf algebra for these differentials: the differentials are not two-sided coderivations for $\Delta$. Instead, deletion-type differentials are right coderivations and contraction-type differentials are left coderivations. Despite this, the Hopf algebra structure -- especially its underlying dg algebra -- still has immediate consequences for homology. If a matroid complex contains a class whose boundary is the unit, one obtains a contracting homotopy implying the complex is acyclic. For the deletion and contraction differentials such classes can be taken to be the classes of $U_{0,1}$ (i.e., a loop) and $U_{1,1}$ (i.e., a coloop), giving the following.

\begin{manualtheorem}{\ref{thm:acyclicity-summary}}
  	All of the following complexes are acyclic:
	\[
		\left(\MC_{\bullet}, \partial_{\del}\right), \quad \quad \left(\MC_{\bullet}, \partial_{\del}^{\tot} \right), 
		\quad \quad \left(\MC_{\bullet}, \partial_{\con} \right), \quad \text{and} \quad \left(\MC_{\bullet}, \partial_{\con}^{\tot} \right).
	\]
\end{manualtheorem}

Put differently, this theorem says that the ability to sum with a loop or coloop results in an acyclic complex. Interestingly, \cite{BBCMMW24,ABBCV26} use a similar ``adjoining a coloop'' trick to prove certain complexes related to the top-weight cohomology of the moduli of principally polarized abelian varieties of dimension $g$ and other locally symmetric varieties are also acyclic. This also shows that the rational homology of the matroid complexes originally defined by \cite{ABGW00} vanishes above degree one. Since \cite{ABGW00} define their complexes over $\ZZ$, there may still be interesting torsion classes. 

We can extend this result to many subcomplexes defined in terms of interesting combinatorial properties of matroids. If $\bP$ is a property of matroids stable under isomorphism, we let $\MC^{\bP}_{\bullet}$ denote the subspace spanned by those classes that have property $\bP$. When $\bP$ is closed under deletion or contraction our differentials restrict to form subcomplexes $(\MC_{\bullet}^{\bP}, \partial)$ where $\partial$ is one of our four differentials depending on whether $\bP$ is deletion or contraction closed. Examples of such properties include being regular, being graphic, being cographic, etc. 
 
 \begin{manualcorollary}{\ref{cor:acyclicity-summary}}
 	Let $\bP$ be a property of matroids stable under isomorphism.
	\begin{enumerate}
	\item If $\bP$ is deletion closed and closed under direct sum with $U_{0,1}$ then $(\MC_{\bullet}^{\bP}, \partial_{\del})$ and $(\MC^{\bP}_{\bullet}, \partial_{\del}^{\tot})$ are acyclic. 
	\item If $\bP$ is contraction closed and closed under direct sum with $U_{1,1}$ then $(\MC_{\bullet}^{\bP}, \partial_{\con})$ and $(\MC^{\bP}_{\bullet}, \partial_{\con}^{\tot})$ are acyclic. 
	\end{enumerate}
 \end{manualcorollary} 

The corollary confirms an observation of Hampe, who noted that computations suggest many combinatorial subcomplexes of the \cite{ABGW00} matroid complexes have trivial homology \cite{hampe17}. This confirmation is again only over $\QQ$.

This acyclicity theorem is best viewed as an organizing principle rather than as bad news. It says that the unreduced matroid complexes are too large: the ability to adjoin a loop or coloop already forces a contracting homotopy. Consequently, any genuinely interesting homology must come from restrictions that break these one-element operations, such as simplicity, looplessness, bounded rank, or connected quotients. In the graph complex setting, the analogous fact is that the interesting homology lives on connected graphs. For matroid complexes, the corresponding reduction is:
 
\begin{manualtheorem}{\ref{thm:connected-generators}}\label{intro:connected}
  The algebra $(\MC_\bullet, \star)$ is the free super-commutative algebra on its connected matroid classes, and its indecomposable quotient is generated by connected matroid classes.  
\end{manualtheorem}
 
In other words, every class decomposes uniquely as a product of connected classes, so the study of $\MC_\bullet$ as an algebra reduces to connected matroids. When the relevant class of matroids is loopless (on the deletion side) or coloopless (on the contraction side), the disconnected classes form a subcomplex, and one can pass to a \emph{connected quotient} that identifies with the indecomposable quotient of the algebra. This is the analogue of the passage from all graphs to connected graphs in graph homology.

Guided by these structural results, we carry out two rounds of computations exploring the deletion matroid complex and several combinatorial subcomplexes. Using the \textit{Macaulay2} database of all matroids on at most nine elements, we compute the full deletion complex and the regular, binary, ternary, simple, and loopless subcomplexes in this range. Note that the number of 10-element matroids, up to isomorphism, remains an open question; however, there are at least $2.5\times 10^{12}$ non-isomorphic matroids of rank 5 on ten elements \cite{mayhewRoyle08}. Figure~\ref{fig:data1} shows $\MC_{\bullet}$ for up to $n=9$.
 
Our computations confirm Theorem~\ref{thm:acyclicity-summary} and Corollary~\ref{cor:acyclicity-summary} for the deletion, regular, binary, and ternary complexes. The simple and loopless deletion complexes are isomorphic (see Proposition~\ref{prop:vanishing-2loops-parallel}) and behave differently: apart from the degree-zero class given by $[\varnothing,1]$ they have a nonzero degree-$1$ class represented by $[U_{1,1},\eta]$, and we found no further homology through $n=8$. However, extending the computation to larger ground sets reveals genuinely new phenomena.
 
In a second round of computation, we use the database of connected simple loopless regular matroids compiled by Fripertinger and Wild \cite{fripertingerWild}, which extends through ground-set size $n=15$. Using linear algebra over $\FF_2$ to test for odd automorphisms in the binary representations, we compute the connected quotient $(\MC_{\bullet}^{\reg,\simp,\cont},\partial_{\del})$, which by Theorem~\ref{thm:connected-generators} and Corollary~\ref{cor:connected-indecomposables} is the indecomposable quotient of the simple loopless regular deletion matroid subcomplex. We computed several nontrivial homology classes generated by $U_{1,1}$ or the graphic matroid of a wheel graph $W_{g}$ of genus $g$ with odd genus. The data lead to the following conjectural description of the homology of the indecomposable quotient of the simple loopless regular matroid complex.

\begin{manualconjecture}{\ref{con:reg-homology}}
    The homology of the indecomposable quotient of the simple loopless regular matroid complex is generated by the classes of $U_{1,1}$ and $\sM(W_{2k+1})$ for all $k\geq1$.
\end{manualconjecture}

Here $\sM(W_{2k+1})$ is the graphic matroid of the wheel graph $W_{2k+1}$ with $4k+2$ edges and genus $2k+1$. The appearance of odd wheels is especially suggestive in light of the graph-complex picture. In the graph complex, the odd wheel graphs $W_{2k+1}$ give nontrivial homology classes \cite{willwacher15} and play important roles in related computations and applications \cite{brown21,CGP21,BCGP24}. Conjecture~\ref{con:reg-homology} may therefore be viewed as a matroidal shadow of the odd-wheel phenomenon in graph homology. 

 \begin{figure}[ht]
\centering
\scalebox{0.9}{
\renewcommand{\arraystretch}{1.5}
\setlength{\tabcolsep}{6pt}
\begin{tikzpicture}
\node[inner sep=0pt] (tbl) {%
\begin{tabular}{r | *{9}{>{\centering\arraybackslash}m{2.2em}} >{\centering\arraybackslash}m{5.2em}}
$9$ & \cellcolor{blue!30} & \cellcolor{blue!30} & \cellcolor{blue!30} & \cellcolor{blue!30} & \cellcolor{blue!30} & \cellcolor{blue!30} & \cellcolor{blue!30} & \cellcolor{blue!30} & \cellcolor{blue!30} & $0$ \\[2pt]
$8$ & \cellcolor{blue!30} & \cellcolor{blue!30} & \cellcolor{blue!30} & \cellcolor{blue!30} & \cellcolor{blue!30} & \cellcolor{blue!30} & \cellcolor{blue!30} & \cellcolor{blue!30} & $0$ & $0$ \\[2pt]
$7$ & \cellcolor{blue!30} & \cellcolor{blue!30} & \cellcolor{blue!30} & \cellcolor{blue!30} & \cellcolor{blue!30} & \cellcolor{blue!30} & \cellcolor{blue!30} & $0$ & $0$ & $0$ \\[2pt]
$6$ & \cellcolor{blue!30} & \cellcolor{blue!30} & \cellcolor{blue!30} & \cellcolor{blue!30} & \cellcolor{blue!30} & \cellcolor{blue!30} & $0$ & $0$ & $0$ & $\mathbb{Q}^{131}$ \\[2pt]
$5$ & \cellcolor{blue!30} & \cellcolor{blue!30} & \cellcolor{blue!30} & \cellcolor{blue!30} & \cellcolor{blue!30} & $0$ & $0$ & $0$ & $\mathbb{Q}^{21}$ & $\mathbb{Q}^{174{,}771}$ \\[2pt]
$4$ & \cellcolor{blue!30} & \cellcolor{blue!30} & \cellcolor{blue!30} & \cellcolor{blue!30} & $0$ & $0$ & $0$ & $\mathbb{Q}^{9}$ & $\mathbb{Q}^{299}$ & $\mathbb{Q}^{174{,}771}$ \\[2pt]
$3$ & \cellcolor{blue!30} & \cellcolor{blue!30} & \cellcolor{blue!30} & $0$ & $0$ & $0$ & $\mathbb{Q}^{2}$ & $\mathbb{Q}^{9}$ & $\mathbb{Q}^{21}$ & $\mathbb{Q}^{131}$ \\[2pt]
$2$ & \cellcolor{blue!30} & \cellcolor{blue!30} & $0$ & $0$ & $0$ & $0$ & $0$ & $0$ & $0$ & $0$ \\[2pt]
$1$ & \cellcolor{blue!30} & $\mathbb{Q}$ & $\mathbb{Q}$ & $0$ & $0$ & $0$ & $0$ & $0$ & $0$ & $0$ \\[2pt]
$0$ & $\mathbb{Q}$ & $\mathbb{Q}$ & $0$ & $0$ & $0$ & $0$ & $0$ & $0$ & $0$ & $0$ \\[2pt]
\midrule
 & \bfseries$0$ & \bfseries$1$ & \bfseries$2$ & \bfseries$3$ & \bfseries$4$ & \bfseries$5$ & \bfseries$6$ & \bfseries$7$ & \bfseries$8$ & \bfseries$9$ \\
\end{tabular}%
};
\node[left=10pt of tbl.west, rotate=90, anchor=center, font=\itshape] {$r$};
\node[below=6pt of tbl.south, anchor=center, font=\itshape] {$n$};
\end{tikzpicture}
}
\caption{Dimensions of the chain groups $\MC_{n-r,r}$ for $0\le r\le n\le 9$. Blue entries indicate bidegrees with no matroids.}
\label{fig:data1}
\end{figure}
 
 %%%%%%%%%%%%%%%%%%%%%%%%%%%%%%%%%%%%%%%%%%%%%%%%%%%
 \subsection*{Outline} 
 %%%%%%%%%%%%%%%%%%%%%%%%%%%%%%%%%%%%%%%%%%%%%%%%%%%
 Section~\ref{sec:background} recalls the necessary background on matroids. Section~\ref{sec:matroid-complex} contains our construction of matroid complexes in their various guises. Section~\ref{sec:comparison} compares our framework with \cite{ABGW00} and \cite{BBCMMW24}. Section~\ref{sec:algebraic-structure} establishes the Hopf algebra structure and proves the acyclicity and free-algebra theorems. Section~\ref{sec:computation} presents our computational work. 
 
%%%%%%%%%%%%%%%%%%%%%%%%%%%%%%%%%%%%%%%%%%%%%%%%%%%
\subsection*{Acknowledgments} 
%%%%%%%%%%%%%%%%%%%%%%%%%%%%%%%%%%%%%%%%%%%%%%%%%%%

We warmly thank Madeline Brandt, Francis Brown, Melody Chan, Daniel Corey, and Dustin Ross for helpful discussions. We also thank Michael Borinsky, Simone Hu, and Bernd Sturmfels for their encouragement to write this note. Finally, we are immensely grateful to Benjamin Ashlock, who worked on this project in its infancy.   

Juliette Bruce was partially supported by the National Science Foundation under Award Nos. NSF FRG DMS-2053221 and NSF MSPRF DMS-2002239. Bailee Zacovic was supported by the National Science Foundation Graduate
Research Fellowship Program under Grant No. DGE-2241144. Initial work on this project took place during the ``Midwest Research Experience for Graduates'' held in Summer 2023 at the University of Michigan, supported by NSF DMS-2317485. We thank the organizers of this program for providing a welcoming and comfortable environment that made these interactions possible. Any opinions,
findings, and conclusions or recommendations expressed in this material are those of the authors and do not reflect the views of the National Science Foundation.

%%%%%%%%%%%%%%%%%%%%%%%%%%%%%%%%%%%%%%%%%%%%%%%%%%%
\section{Background}\label{sec:background}
%%%%%%%%%%%%%%%%%%%%%%%%%%%%%%%%%%%%%%%%%%%%%%%%%%%

In this section, we briefly recall the basic facts concerning matroids that we will use and fix our notation and terminology. For a thorough treatment of matroid theory, see \cite{oxley11}. Throughout, we let $[n]\coloneqq \{1,2,\ldots,n\}$ and write $\binom{[n]}{k}$ for the set of $k$-element subsets of $[n]$.  There are numerous equivalent ways to define matroids; we employ the definition via bases. 

\begin{definition}
A \emph{matroid} is a pair $\sM=(E,\mathcal{B})$, where $E$ is a finite set and $\mathcal{B}$ is a collection of subsets of $E$ satisfying the following two axioms:
\begin{enumerate}
    \item[(B1)] $\mathcal{B}$ is nonempty.
    \item[(B2)] Let $S,T \in \cB$ be distinct subsets. For every $x\in S\setminus T$ there exists $y\in T\setminus S$ such that $(S\setminus \{x\})\cup \{y\}$ is in $\cB$. 
\end{enumerate}
\end{definition}
The (B2) axiom is commonly called the basis exchange axiom of a matroid. Given a matroid $\sM=(E,\mathcal{B})$ we call the finite set $E$ the \emph{ground set} of $\sM$ and the set $\cB$ the \emph{set of bases} of $\sM$. To avoid confusion, we will often write $E(\sM)$ and $\cB(\sM)$ for the ground set and bases of a matroid $\sM$. An element $S\in \cB$ is called a basis of $\sM$ and a subset $T \subset E$ is said to be an \emph{independent set} of $\sM$ if there exists a basis $S\in \cB$ such that $T\subset S$. We write $\cI(\sM)$ for the set of all independent sets of a matroid $\sM$. The independent sets maximal with respect to inclusion are precisely the bases, and every basis has the same size. 

\begin{definition}
    The \emph{rank} of a matroid $\sM=(E,\mathcal{B})$, which we denote by $\rk(\sM)$, is the size of any basis of $\sM$. The \emph{nullity} of a matroid $\sM$ is equal to $\nul(\sM)\coloneqq|E(\sM)|-\rk(\sM)$.
\end{definition}

\begin{example}\label{exm:into-uniform-matroid}
    Let $0\leq r\leq n$ be non-negative integers. The matroid with ground set $[n]$ whose bases are all subsets of $[n]$ of size $r$ is called the \emph{uniform matroid} of rank $r$ on $n$ elements. We denote this matroid by $U_{r,n}$, and it has nullity $n-r$.
\end{example}

While generally we will speak of independent sets and bases, at times it is also useful to discuss the dual notions. Let $\sM=(E,\mathcal{B})$ be a matroid. A subset $S\subset E$ is a \emph{dependent set} of $\sM$ if it is not an independent set of $\sM$. The minimal dependent sets are the \emph{circuits} of $\sM$. Reversing the roles of independent and dependent sets provides a notion of duality.

\begin{definition}
    The \textit{dual} of a matroid $\sM$ is the matroid $\sM^*$ whose ground set is $E(\sM)$ and whose bases are $\cB(\sM^*)=\{E(\sM)\setminus S \;\; | \;\; S \in \cB(\sM) \}$.
\end{definition}

It is immediate from the definition that $(\sM^*)^*=\sM$. Further, duality flips the roles of rank and nullity in the sense that $\rk(\sM^*)=\nul(\sM)$ and $\nul(\sM^*)=\rk(\sM)$. 

\begin{definition}\label{def:matroidpotpourri}
Fix a matroid $\sM=(E,\mathcal{B})$.
    \begin{itemize}
        \item An element $x\in E$ is a \emph{loop} of $\sM$ if $x\notin B$ for all $B\in \mathcal{B}$.
        \item An element $x\in E$ is a \emph{coloop} of $\sM$ if $x\in B$ for all $B\in \mathcal{B}$.
        \item Two elements $x,y\in E$ are \emph{parallel} if $\{x,y\}$ is a circuit.
        \item Two elements $x,y\in E$ are in \emph{series} if every basis contains $x$ or $y$ and $\{x,y\}$ is minimal with respect to this property. 
    \end{itemize}
\end{definition}

A matroid is \emph{simple} if it is loopless and has no parallel elements (i.e., it has no circuits of size one or two). A matroid is \emph{cosimple} if it has no coloops and no elements in series. An \emph{isomorphism} from a matroid $\sM$ to a matroid $\sM'$ is a bijection of ground sets $\psi:E(\sM)\to E(\sM')$ which preserves bases, i.e., given $S\subset E(\sM)$ we have that $S \in \cB(\sM)$ if and only if $\psi(S) \in \cB(\sM')$. Most matroid properties, e.g., rank, loops, coloops, etc., are preserved by isomorphisms.

%%%%%%%%%%%%%%%%%%%%%%%%%%%%%%%%%%%%%%%%%%%%%%%%%%%
\subsection{Operations on Matroids}
%%%%%%%%%%%%%%%%%%%%%%%%%%%%%%%%%%%%%%%%%%%%%%%%%%%

There are four natural operations one may define on matroids: deletion, contraction, restriction, and direct sum. The first two operations will be used to define matroid complexes in the following section. 

\begin{definition}
    Let $\sM$ be a matroid and fix a subset $S \subset E(\sM)$.
    \begin{enumerate}
        \item The \emph{deletion} of $S$ from $\sM$ is the matroid $\sM\setminus S$ whose ground set is $E(\sM)\setminus S$ and whose independent sets are:
            $
            \cI(\sM\setminus S) \coloneqq \left\{T\subset E(\sM)\setminus S \;\;\big|\;\; T\in \cI(\sM) \right\}.
            $
        \item The \emph{contraction} of $\sM$ by $S$ is the matroid $\sM/S$ whose ground set is $E(\sM)\setminus S$ and whose independent sets are:
            \[
            \cI(\sM/S) \coloneqq \left\{T\subset E(\sM)\setminus S \;\;\Bigg|\;\; \begin{matrix}
            \text{$\exists S'\subset S$ maximal among independent sets}\\
            \text{of $\sM$ in $S$ such that $T\cup S' \in \cI(\sM)$}
            \end{matrix}
            \right\}.
            \]
        \item The \emph{restriction} of $\sM$ to $S$ is the matroid $\sM|S$ whose ground set is $S$ and whose independent sets are:
            $
            \cI(\sM|S) \coloneqq \left\{T\subset S \;\;\big|\;\; T\in \cI(\sM) \right\}.
            $
    \end{enumerate}
\end{definition}

Abusing notation slightly, we will often write $\sM\setminus x$ and $\sM/x$ to mean $\sM\setminus\{x\}$ and $\sM/\{x\}$ respectively when deleting or contracting a single element $x\in E(\sM)$. Deletion and contraction are dual in the sense that $(\sM/S)^*=\sM^*\setminus S$. Any matroid that arises from $\sM$ via a sequence of deletions or contractions is called a \emph{minor} of $\sM$. The following lemma shows how deletion and contraction affect the rank of a matroid, which will be useful in later sections. 

 \begin{lemma}\label{lem:rank-deletion-contraction}
    Let $\sM=(E,\cB)$ be a matroid and $x\in E$.
    \begin{enumerate}
        \item If $x$ is not a coloop then $\rk(\sM\setminus x)=\rk(\sM)$, and if $x$ is a coloop $\rk(\sM\setminus x)=\rk(\sM)-1$. 
        \item If $x$ is a loop then $\rk(\sM/x)=\rk(\sM)$, and if $x$ is not a loop then $\rk(\sM/x)=\rk(\sM)-1$. 
    \end{enumerate}
\end{lemma}

The last operation on matroids is the direct sum of two matroids. 

\begin{definition}
    The \emph{direct sum} of two matroids $\sM=(E,\cB)$ and $\sM'=(E',\cB')$ is the matroid $\sM\oplus \sM'$ whose ground set is $E(\sM\oplus \sM')\coloneqq E\sqcup E'$ and whose bases are
    \[
    \cB(\sM\oplus \sM')\coloneqq \left\{ B \sqcup B' \;\; |\;\; B \in \cB(\sM), B' \in \cB(\sM') \right\}.
    \]
\end{definition}

\begin{example}\label{exm:adjoin-loop}
    A special case of direct sum that will appear later is $\sM\oplus U_{0,1}$ where $U_{0,1}$ is the uniform matroid of rank 0 on one element. Writing $\{\ell\}$ for $E(U_{0,1})$ we see that $E(\sM\oplus U_{0,1})=E(\sM)\sqcup\{\ell\}$, and further, $\ell$ is not contained in any basis, so it is a loop of both $U_{0,1}$ and $\sM\oplus U_{0,1}$. For this reason, direct summing with $U_{0,1}$ is often referred to as adjoining a loop. 
\end{example}
%%%%%%%%%%%%%%%%%%%%%%%%%%%%%%%%%%%%%%%%%%%%%%%%%%%
\subsection{Sources of Matroids}
%%%%%%%%%%%%%%%%%%%%%%%%%%%%%%%%%%%%%%%%%%%%%%%%%%%

Given a (connected undirected) graph $G=(V,E)$, the graphic matroid $\sM(G)$ has ground set $E$ and independent sets the forests of $G$; its bases are the spanning trees. Matroids isomorphic to $\sM(G)$ are \emph{graphic}, and those isomorphic to $\sM(G)^*$ are \emph{cographic}. If $A$ is a matrix over a field $\FF$, its column matroid $\sM(A)$ has ground set the columns of $A$ and bases the maximal linearly independent subsets. A matroid is representable over $\FF$ if it is isomorphic to $\sM(A)$ for some $A$. A matroid is \emph{regular} if it is representable over every field, equivalently, if it admits a totally unimodular real representation. In particular, every graphic or cographic matroid is regular.

%%%%%%%%%%%%%%%%%%%%%%%%%%%%%%%%%%%%%%%%%%%%%%%%%%%
\section{Matroid Complexes}\label{sec:matroid-complex}
%%%%%%%%%%%%%%%%%%%%%%%%%%%%%%%%%%%%%%%%%%%%%%%%%%%

Matroid complexes were introduced in \cite{ABGW00} as a direct generalization of Kontsevich's graph complexes. This section presents these complexes using a slightly different construction. Many of the results below appear in \cite{ABGW00}. However, to keep this note self-contained -- and in the hope of encouraging further interest in matroid complexes -- we provide a new exposition. 

We work with a single bigraded vector space $\MC_{\bullet,\bullet}\coloneqq \bigoplus_{k,r} \MC_{k,r}$ spanned by classes of matroids with an orientation on their ground set, where $k$ is the nullity of the matroid and $r$ is the rank. Matroid deletion and contraction allow one to define differentials on $\MC_{\bullet,\bullet}$. For them to respect the bigrading we must in fact split these into four natural operators: $\partial_{\del},\partial_{\clp}, \partial_{\con}$, and $\partial_{\lp}$.  The maps $\partial_{\del}$ and $\partial_{\clp}$ are defined by deleting elements that are not coloops and deleting coloops respectively. Similarly, $\partial_{\con}$ and $\partial_{\lp}$ are defined by contracting elements that are not loops and contracting loops respectively. Both $\partial_{\del}$ and $\partial_{\lp}$ have bidegree $(-1,0)$, while $\partial_{\con}$ and $\partial_{\clp}$ have bidegree $(0,-1)$. These differentials assemble into two bicomplexes: the deletion bicomplex $(\MC_{\bullet,\bullet}, \partial_{\del},\partial_{\clp})$ and the contraction bicomplex $(\MC_{\bullet,\bullet}, \partial_{\lp}, \partial_{\con})$. Matroid duality exchanges these two bicomplexes after interchanging the nullity and rank gradings. 

\subsection{The Vector Space $\MC$}

Throughout, given a set $S$, we let $\QQ\langle S\rangle $ and $\ZZ\langle S\rangle$ denote the $\QQ$-vector space and free $\ZZ$-module, respectively, generated by elements of $S$. 

\begin{definition}
    An \emph{orientation} on a matroid $\sM=(E,\cB)$ is a primitive element of $\bigwedge^{|E|}\ZZ\left\langle E \right\rangle$.
\end{definition}

Every matroid, even the empty matroid $\varnothing$, has an orientation, since $\bigwedge^{|\varnothing|}\ZZ\langle \varnothing\rangle \cong \ZZ$. We will normally refer to $1$ as the canonical orientation on $\varnothing$.

\begin{remark}
This notion of orientation should not be confused with the usual notion of an oriented matroid; here the orientation data is an orientation of the determinant line of the ground set.
\end{remark}

Let $\widetilde{\cM}$ denote the (ungraded) $\QQ$-vector space freely generated by pairs $(\sM,\eta)$ where $\sM$ is any matroid and $\eta$ is any orientation on $\sM$. There is a natural bigrading on $\widetilde{\cM}$ by nullity and rank. Given integers $k$ and $r$, let
\[
\widetilde{\cM}_{k,r} \coloneqq \QQ\left\langle \left(\sM,\eta\right) \;\; \Bigg|\;\;
\begin{matrix}
    \text{$\sM=(E,\cB)$ is a matroid with $\nul(\sM)=k$, $\rk(\sM)=r$,}\\
    \text{and $\eta$ an orientation on $\sM$}
\end{matrix}\;
\right\rangle \subset \widetilde{\cM}.
\]
This turns $\widetilde{\cM}$ into a $\ZZ^2$-graded $\QQ$-vector space with $\widetilde{\cM}_{\bullet,\bullet}\cong \bigoplus_{(k,r)\in \ZZ^2}\widetilde{\cM}_{k,r}$. 

We wish to quotient $\widetilde{\cM}$ by the natural equivalence relations that arise from reversing orientation and identifying isomorphic matroids. Before we can do this, we must note how orientations push forward along matroid isomorphisms. Since a matroid isomorphism $\psi:\sM\to \sM'$ is a bijection $\psi:E(\sM)\to E(\sM')$, it induces an isomorphism 
\[
\begin{tikzcd}[column sep = 3em]
\psi_{*}:\displaystyle \bigwedge^{|E(\sM)|}\ZZ\left\langle E(\sM) \right\rangle \arrow[r,"\sim"]& \displaystyle\bigwedge^{|E(\sM')|}\ZZ\left\langle E(\sM') \right\rangle.
\end{tikzcd}
\]

\begin{definition}
	Let $\sM=(E,\cB)$ be a matroid and $\eta\in\left(\bigwedge^{|E|}\ZZ\langle E \rangle\right)^{\times}$ an orientation. The \emph{pushforward} of $\eta$ along an isomorphism $\psi:\sM\to \sM'$ is the image of $\eta$ under the map $\psi_{*}$.
\end{definition}

Consider the subspace spanned by the following relations:
\[
\widetilde{\cR}\coloneqq \Span_{\QQ}\left\langle
\begin{matrix}
    (\sM,\eta)+(\sM,-\eta) \\
(\sM,\eta)-(\sM',\psi_{*}(\eta))\end{matrix} \;\; \bigg| \;\;
\begin{matrix}
    \text{all pairs $(\sM,\eta)$ of a matroid with an orientation}\\
    \text{all matroid isomorphisms $\psi:\sM\to\sM'$}
\end{matrix}
\right\rangle \subset \widetilde{\cM}.
\]
Let $\MC$ be the (ungraded) $\QQ$-vector space $\widetilde{\cM}/\widetilde{\cR}$. The relations generating $\widetilde{\cR}$ are homogeneous in the nullity-rank grading, and so the $\ZZ^2$-grading descends to $\MC$, given by
\[
\MC \cong \bigoplus_{(k,r)\in \ZZ^2} \MC_{k,r} \cong \bigoplus_{(k,r)\in \ZZ^2} \widetilde{\cM}_{k,r}/\widetilde{\cR}_{k,r}.
\]
We may think of $\MC_{k,r}$ as spanned by classes $[\sM,\eta]$ where $\sM$ is a matroid of rank $r$ with nullity $k$ and $\eta$ is an orientation on $\sM$. Note such a class has ground set size $k+r$ and $\MC_{\bullet,\bullet}$ is supported in the first quadrant. 

%%%%%%%%%%%%%%%%%%%%%%%%%%%%%%%%%%%%%%%%%%%%%%%%%%%
\subsection{Deletion \& Contraction Differentials}
%%%%%%%%%%%%%%%%%%%%%%%%%%%%%%%%%%%%%%%%%%%%%%%%%%%

We now define differentials on $\MC_{\bullet,\bullet}$. There are four natural choices of differentials one could consider: two coming from deletion and two coming from contraction. In both cases the key is that an orientation $\eta$ on a matroid $\sM$ induces orientations on both $\sM\setminus x$ and $\sM/x$ for any element $x\in E(\sM)$. In this section we focus on what we view as the more interesting differentials that arise from deleting or contracting elements that are not coloops or loops, respectively. 

If $S$ is a finite set and $x \in S$, interior product with respect to $x$ is the linear map
\[
\begin{tikzcd}[column sep = 3em]
\iota_{x}:\displaystyle \bigwedge^{|S|}\ZZ\left\langle S \right\rangle \arrow[r]& \displaystyle\bigwedge^{|S|-1}\ZZ\left\langle S \right\rangle\end{tikzcd}
\]
given by $\iota_{x}(x \wedge \alpha)=\alpha$ and extended with the usual sign convention. The image of this map is contained in the free submodule $\bigwedge^{|S|-1}\ZZ\left\langle S\setminus\{x\} \right\rangle$, and thus, we will view $\iota_{x}$ as a linear map into this subspace. The interior product preserves primitive elements and as such we can use it to define the deletion and contraction of orientations on matroids. Note that both $\sM\setminus x$ and $\sM/x$ have ground set $E(\sM)\setminus\{x\}$, so the induced orientations are defined by the same interior product operator $\iota_x$.

\begin{definition}\label{def:orientation}
    Let $\sM$ be a matroid, $\eta$ an orientation on $\sM$, and $x\in E(\sM)$. 
    \begin{enumerate}
        \item The \emph{induced deletion orientation} on $\sM\setminus x$ is $\eta\setminus x\coloneqq \iota_{x}(\eta)$. 
        \item The \emph{induced contraction orientation} on $\sM/x$ is $\eta/x\coloneqq \iota_{x}(\eta)$.
    \end{enumerate}
\end{definition}

\begin{example}\label{ex:K4-1}
As a running example consider the graphic matroid $\sM(K_4)$ on the complete graph with four vertices. We choose an ordering on the edges of $K_{4}$, which in turn gives an identification $E(\sM(K_4)) \cong [6]$. Fix the orientation $\eta\coloneqq 1\wedge 2\wedge 3\wedge 4\wedge 5\wedge 6$. Since $K_4$ is connected with 4 vertices and 6 edges, the matroid $\sM(K_4)$ has rank 3 and nullity 3. Hence $[\sM(K_4),\eta]\in \MC_{3,3}$. Moreover, this class is nonzero, as one can check that every automorphism of $\sM(K_4)$ induces an even permutation on $[6]$, and so no automorphism sends $\eta$ to $-\eta$. 

For later use, given $i \in [6]$ we let $\eta_i\coloneqq (-1)^{i-1}\,1\wedge \cdots \wedge \widehat{i}\wedge \cdots \wedge 6$. With this notation we have that $\eta\setminus i=\eta/i=\eta_i$ for all $i \in [6]$. For example, we have that
\[
\eta\setminus 1=2\wedge 3\wedge 4\wedge 5\wedge 6,\qquad
\eta\setminus 2=-1\wedge 3\wedge 4\wedge 5\wedge 6,\qquad
\eta\setminus 3=1\wedge 2\wedge 4\wedge 5\wedge 6.
\]
\end{example}

Define linear maps $\tilde{\partial}_{\del}:\widetilde{\cM}_{\bullet,\bullet}\to \widetilde{\cM}_{\bullet, \bullet}$ and $\tilde{\partial}_{\con}:\widetilde{\cM}_{\bullet,\bullet}\to \widetilde{\cM}_{\bullet, \bullet}$ on basis elements by:
 \[
 \tilde{\partial}_{\del}\left((\sM,\eta)\right) \coloneqq \sum_{\substack{x\in E(\sM) \\ \text{$x$ not a coloop}}} \left(\sM\setminus x, \eta \setminus x\right) \quad \quad \text{and} \quad \quad \tilde{\partial}_{\con}\left((\sM,\eta)\right) \coloneqq \sum_{\substack{x\in E(\sM) \\ \text{$x$ not a loop}}} \left(\sM/x, \eta/x\right),
 \]
 and extended linearly. Lemma~\ref{lem:rank-deletion-contraction} implies that $\tilde{\partial}_{\del}$ preserves rank and decreases nullity by one while $\tilde{\partial}_{\con}$ decreases rank by one and preserves nullity. Thus, $\tilde{\partial}_{\del}$ is a degree $(-1,0)$ linear map and $\tilde{\partial}_{\con}$ is a degree $(0,-1)$ linear map.
  
 We now show that these maps descend to differentials on $\MC$. Before checking the details, it is helpful to separate the two issues. First, the maps $\tilde{\partial}_{\del}$ and $\tilde{\partial}_{\con}$ must respect the relations defining $\MC$; this is a functorial statement and comes from the fact that interior product commutes with relabeling of the ground set. Second, once these maps descend, the reason they square to zero is the usual anti-commutativity of interior products together with the fact that deleting or contracting two distinct elements does not depend on the order. The next few lemmas formalize these two points.

\begin{lemma}\label{lem:deletion-desc}
	Both $\tilde{\partial}_{\del}$ and $\tilde{\partial}_{\con}$ descend to linear maps $\partial_{\del}$ and $\partial_{\con}$ on $\MC$.
\end{lemma}

Before proving this, we need one technical lemma.

\begin{lemma}\label{lem:pushforward-del-con}
	Let $\psi:\sM\to \sM'$ be an isomorphism of matroids and $x\in E(\sM)$. If $\eta$ is an orientation on $\sM$ then 
	\[
	\psi_{*}(\eta)\setminus \psi(x) = \tilde{\psi}_{*}(\eta\setminus x)
	\quad\text{and}\quad
	\psi_{*}(\eta)/\psi(x) = \tilde{\psi}_{*}(\eta/x),
	\]
	where $\tilde{\psi}:E(\sM)\setminus\{x\}\to E(\sM')\setminus\{\psi(x)\}$ is a bijection given by restricting $\psi$.
	\end{lemma}

\begin{proof}
	Since the deletion and contraction of an orientation is defined in terms of interior product maps, both claims follow from the following fact: If $\psi:E\to E'$ is a bijection of finite sets and $x\in E$, the following diagram commutes:
	\[
	\begin{tikzcd}[column sep = 3em, row sep = 3em]
	\displaystyle \bigwedge^{|E|}\ZZ\left\langle E \right\rangle \arrow[r, "\iota_x"] \arrow[d,"\psi_{*}"]& \displaystyle\bigwedge^{|E|-1}\ZZ\left\langle E\setminus\{x\} \right\rangle \arrow[d,"\tilde{\psi}_{*}"]\\
	\displaystyle \bigwedge^{|E'|}\ZZ\left\langle E' \right\rangle \arrow[r, "\iota_{\psi(x)}"]& \displaystyle\bigwedge^{|E'|-1}\ZZ\left\langle E'\setminus\{\psi(x)\} \right\rangle 
	\end{tikzcd}
	\]
	The claim itself can be checked directly on generators. 
\end{proof}

\begin{proof}[Proof of Lemma~\ref{lem:deletion-desc}]
	We must check that $\tilde{\partial}_{\del}(\widetilde{\cR}) \subset \widetilde{\cR}$ and $\tilde{\partial}_{\con}(\widetilde{\cR}) \subset \widetilde{\cR}$. It is enough to check these claims on the two types of generators of $\widetilde{\cR}$. Let $\sM$ be a matroid and $\eta$ an orientation on $\sM$.  Considering the generator $(\sM,\eta)+(\sM,-\eta)$ we compute directly 
	\begin{align*}
	\tilde{\partial}_{\del}\left((\sM,\eta)+(\sM,-\eta)\right) &= \tilde{\partial}_{\del}\left((\sM,\eta)\right) + \tilde{\partial}_{\del}\left((\sM,-\eta)\right) = \sum_{\substack{x\in E(\sM) \\ \text{$x$ not a coloop}}} \left(\sM\setminus x, \eta \setminus x\right) + \sum_{\substack{x\in E(\sM) \\ \text{$x$ not a coloop}}} \left(\sM\setminus x, (-\eta) \setminus x\right) 
	\intertext{since $(-\eta)\setminus x$ is equal to $\iota_{x}(-\eta)$ and $\iota_{x}$ is a linear map we may pull out the minus sign on the right-hand side and recombine to get}
	&=\sum_{\substack{x\in E(\sM) \\ \text{$x$ not a coloop}}} \left(\left(\sM\setminus x, \eta \setminus x\right)+\left(\sM\setminus x, -(\eta \setminus x)\right)\right),
	\end{align*}
	which is in $\widetilde{\cR}$ since each term in the sum is in $\widetilde{\cR}$. The argument for $\tilde{\partial}_{\con}$ is similar noting that $(-\eta)/x$ is equal to $\iota_{x}(-\eta)$. Turning to the second type of generator let $\psi:\sM\to \sM'$ be a matroid isomorphism and consider the generator $(\sM,\eta)-(\sM',\psi_*(\eta))$. Note for any $x\in E(\sM)$ the restriction of $\psi$ induces an isomorphism $\tilde{\psi}: \sM\setminus x \to \sM'\setminus \psi(x)$. Further, $x$ is a coloop of $\sM$ if and only if $\psi(x)$ is a coloop of $\sM'$. Using these two facts we get that:
	\begin{align*}
	\tilde{\partial}_{\del}\left((\sM,\eta)-(\sM',\psi_*(\eta))\right) &= \tilde{\partial}_{\del}\left((\sM,\eta)\right) - \tilde{\partial}_{\del}\left((\sM',\psi_*(\eta))\right) \\
	&= \sum_{\substack{x\in E(\sM) \\ \text{$x$ not a coloop}}} \left(\sM\setminus x, \eta \setminus x\right) - \sum_{\substack{y \in E(\sM') \\ \text{$y$ not a coloop}}} \left(\sM'\setminus y, \psi_*(\eta)\setminus y\right) \\
	&=\sum_{\substack{x\in E(\sM) \\ \text{$x$ not a coloop}}}\left( \left(\sM\setminus x, \eta \setminus x\right) - \left(\sM'\setminus \psi(x), \psi_*(\eta)\setminus \psi(x) \right)\right).
	\intertext{Applying Lemma~\ref{lem:pushforward-del-con} gives}
	&=\sum_{\substack{x\in E(\sM) \\ \text{$x$ not a coloop}}}\left( \left(\sM\setminus x, \eta \setminus x\right) - \left(\sM'\setminus \psi(x), \tilde{\psi}_*(\eta\setminus x)\right)\right),
	\end{align*}
	which again is in $\widetilde{\cR}$ since each term in the sum is in $\widetilde{\cR}$. The same argument works for contraction since $\tilde{\psi}$ is also an isomorphism between $\sM/x$ and $\sM'/\psi(x)$ and $x$ is a loop of $\sM$ if and only if $\psi(x)$ is a loop of $\sM'$. 
	\end{proof}
	
Lastly, we check that although $\tilde{\partial}_{\del}$ and $\tilde{\partial}_{\con}$ are not differentials on $\widetilde{\cM}$ they descend to differentials $\partial_{\del}$ and $\partial_{\con}$ on $\MC_{\bullet,\bullet}$.

\begin{lemma}\label{lem:horizontal-diffs}
	The maps $\partial_{\del}$ and $\partial_{\con}$ are differentials of degree $(-1,0)$ and $(0,-1)$ respectively. 
\end{lemma}

\begin{proof}
	It is enough to show $\tilde{\partial}_{\del}^2(\widetilde{\cM})\subset \widetilde{\cR}$ and $\tilde{\partial}_{\con}^2(\widetilde{\cM})\subset \widetilde{\cR}$. Beginning with the deletion differential, let $\sM$ be a matroid and $\eta$ an orientation on $\sM$. Expanding the definition and using linearity gives:
	\begin{align*}
	\tilde{\partial}_{\del}^2\left((\sM,\eta)\right)
	&=\sum_{\substack{x\in E(\sM) \\ \text{$x$ not a coloop}}}\;\;\sum_{\substack{y\in E(\sM)\setminus\{x\}\\ \text{$y$ not a coloop in }\sM\setminus x}}
	\left((\sM\setminus x)\setminus y,\,(\eta\setminus x)\setminus y\right).
	\end{align*}
	Consider an ordered pair $(x,y)$ appearing in the above sum. Since $y$ is not a coloop in $\sM\setminus x$, there exists a basis $B$ of $\sM\setminus x$ not containing $y$. Note $B$ is also a basis in $\sM$ and it does not contain both $x$ and $y$. This means that $y$ is not a coloop in $\sM$ and $x$ is not a coloop in $\sM\setminus y$. In particular, the swapped ordered pair $(y,x)$ must also appear in the above sum. 
	For $x\neq y$ the interior products $\iota_{x}$ and $\iota_{y}$ anti-commute giving
	\[
	(\eta\setminus x)\setminus y \;=\; \iota_y\iota_x(\eta)\;=\;-\iota_x\iota_y(\eta)\;=\;-\big((\eta\setminus y)\setminus x\big).
	\]
Further, matroid deletion is commutative, meaning for distinct $x,y\in E(\sM)$ both $(\sM\setminus x)\setminus y$ and $(\sM\setminus y)\setminus x$ are equal to $\sM\setminus\{x,y\}$. Thus, we may combine the two summands corresponding to $(x,y)$ and $(y,x)$ to get 
	\[
	\left(\sM\setminus\{x,y\}, (\eta\setminus x)\setminus y\right)+\left(\sM\setminus\{x,y\},- (\eta\setminus x)\setminus y\right)\ \in \widetilde{\cR}.
	\]
	Since every ordered pair $(x,y)$ appearing in the sum is paired with its swap $(y,x)$, it follows that $\tilde{\partial}_{\del}^2\left((\sM,\eta)\right)$ is a sum of elements of $\widetilde{\cR}$, and hence is in $\widetilde{\cR}$ itself. 	
	
	 The contraction case is analogous, using: (i) Given distinct elements $x,y \in E(\sM)$ then $x$ is not a loop in $\sM$ and $y$ is not a loop in $\sM/x$ if and only if $y$ is not a loop in $\sM$ and $x$ is not a loop in $\sM/y$, (ii) The contractions $(\sM/x)/y$ and $(\sM/y)/x$ are equal, and (iii) Since interior products anti-commute, $(\eta/x)/y = -(\eta/y)/x$.
\end{proof}

\begin{example}\label{ex:K4-2}
Continuing the example of $\sM(K_4)$ from Example~\ref{ex:K4-1} notice that no element of this matroid is a loop or a coloop, and so all six terms appear in the deletion and contraction differentials. Moreover, by edge-transitivity of $K_4$, all deletions $\sM(K_4)\setminus i$ are isomorphic, and the same is true for all contractions $\sM(K_4)/i$. From this one computes that both $\partial_{\del}([\sM(K_4),\eta])$ and $\partial_{\con}([\sM(K_4),\eta])$ are zero. Note this agrees with our computations in Figure~\ref{fig:data1}, which shows $\MC_{5}=0$.
\end{example}

%%%%%%%%%%%%%%%%%%%%%%%%%%%%%%%%%%%%%%%%%%%%%%%%%%%
\subsection{Coloop \& Loop Differentials}
%%%%%%%%%%%%%%%%%%%%%%%%%%%%%%%%%%%%%%%%%%%%%%%%%%%
We now define the remaining two differentials on \(\MC_{\bullet,\bullet}\), namely those obtained by deleting coloops and contracting loops. Define the coloop map $\tilde{\partial}_{\clp}$ and loop map $\tilde{\partial}_{\lp}$ to be linear maps $\widetilde{\cM} \to \widetilde{\cM}$ defined on basis elements by:
 \[
 \tilde{\partial}_{\clp}\left((\sM,\eta)\right) \coloneqq \sum_{\substack{x\in E(\sM) \\ \text{$x$ is a coloop}}} \left(\sM\setminus x, \eta \setminus x\right) \quad \quad \text{and} \quad \quad \tilde{\partial}_{\lp}\left((\sM,\eta)\right) \coloneqq \sum_{\substack{x\in E(\sM) \\ \text{$x$ is a loop}}} \left(\sM/x, \eta/x\right),
 \]
and extended linearly. Lemma~\ref{lem:rank-deletion-contraction} shows that these maps have degrees $(0,-1)$ and $(-1,0)$, respectively. Arguments similar to those in the previous section show that $\tilde{\partial}_{\clp}$ and  $\tilde{\partial}_{\lp}$ descend to differentials.

\begin{lemma}\label{lem:vertical-diffs}
	The maps $\tilde{\partial}_{\clp}$ and $\tilde{\partial}_{\lp}$ descend to differentials $\partial_{\clp}$ and $\partial_{\lp}$ on $\MC_{\bullet,\bullet}$, of bidegrees $(0,-1)$ and $(-1,0)$ respectively. 
\end{lemma}

%%%%%%%%%%%%%%%%%%%%%%%%%%%%%%%%%%%%%%%%%%%%%%%%%%%
\subsection{The Matroid Bicomplexes}
%%%%%%%%%%%%%%%%%%%%%%%%%%%%%%%%%%%%%%%%%%%%%%%%%%%
Having defined four differentials -- $\partial_{\del}, \partial_{\con}, \partial_{\clp}$, and $\partial_{\lp}$ -- we turn to explaining how they can be assembled into four bicomplexes:
\[
(\MC_{\bullet,\bullet}, \partial_{\del}, \partial_{\clp}), \quad \quad  (\MC_{\bullet,\bullet}, \partial_{\lp}, \partial_{\con}), \quad \quad (\MC_{\bullet,\bullet}, \partial_{\del}, \partial_{\con}), \quad \text{and} \quad  (\MC_{\bullet,\bullet}, \partial_{\lp}, \partial_{\clp}).
\]
We will be most interested in the first two, which we call the \emph{(matroid) deletion} and \emph{(matroid) contraction bicomplexes}.  We use the convention that the first differentials are horizontal, the second are vertical,  and that squares in a bicomplex anti-commute. 

\begin{proposition}
	Both  $(\MC_{\bullet,\bullet}, \partial_{\del}, \partial_{\clp})$ and  $(\MC_{\bullet,\bullet}, \partial_{\lp}, \partial_{\con})$ are bicomplexes.
\end{proposition}

\begin{proof}
	By Lemma~\ref{lem:horizontal-diffs} and Lemma~\ref{lem:vertical-diffs} we know in both cases the horizontal and vertical maps give chain complexes. All that remains is to check the differentials anti-commute. Consider the ``total deletion'' operator on $\widetilde{\cM}_{\bullet,\bullet}$ given by $\tilde{\partial}_{\del}^{\tot} \coloneqq \tilde{\partial}_{\del}+\tilde{\partial}_{\clp}$. The same arguments used in the proofs of Lemmas~\ref{lem:deletion-desc} and \ref{lem:horizontal-diffs} show that $\tilde{\partial}_{\del}^{\tot}$ descends to a linear map $\partial_{\del}^{\tot}$ on $\MC$ such that $(\partial_{\del}^{\tot})^2=0$. Expanding this relation gives
	\[
	0=(\partial_{\del}+\partial_{\clp})^2
	=\partial_{\del}^2+\big(\partial_{\del}\partial_{\clp}+\partial_{\clp}\partial_{\del}\big)+\partial_{\clp}^2.
	\]
	Since $\partial_{\del}^2=\partial_{\clp}^2=0$ (see Lemmas~\ref{lem:horizontal-diffs} and \ref{lem:vertical-diffs}) we conclude, as needed, that $\partial_{\del}\partial_{\clp}+\partial_{\clp}\partial_{\del}=0$. This verifies that $(\MC_{\bullet,\bullet},\partial_{\del},\partial_{\clp})$ is a bicomplex. The contraction case is analogous.
\end{proof} 

The deletion and contraction matroid bicomplexes constructed above are related via matroid duality. Since the ground sets of $\sM$ and $\sM^*$ are the same, an orientation on $\sM$ is an orientation on $\sM^*$ and vice versa. Therefore, there is a linear map of (ungraded) vector spaces $(-)^*:\widetilde{\cM} \to \widetilde{\cM}$ given by $(\sM,\eta)\mapsto (\sM^*,\eta)$. The fact that $(\sM^*)^*=\sM$ implies that this is an isomorphism. Further, one may check that the subspace $\widetilde{\cR}$ is stable under taking duals and so duality gives an isomorphism $(-)^*:\MC \to \MC$ of \emph{ungraded} vector spaces. This actually extends to the bicomplexes we have constructed. However, since duality swaps rank and nullity the induced isomorphism does not strictly preserve the grading, but instead transposes it. Making this precise, a \emph{grading-transposing morphism} of bicomplexes $\phi:(C_{\bullet,\bullet},d_{h},d_{v}) \to (D_{\bullet,\bullet},\partial_{h},\partial_{v})$ is a morphism $\phi:C_{\bullet,\bullet} \to D_{\bullet,\bullet}$ such that $\phi(C_{k,r}) \subset D_{r,k}$ for all $k,r \in \ZZ$,  $\phi\circ d_{h} = \partial_{v} \circ \phi$, and $\phi\circ d_{v} = \partial_{h} \circ \phi$.  A \emph{grading-transposing isomorphism} is a grading-transposing morphism that is an isomorphism on the underlying vector spaces. 

\begin{proposition}\label{prop:duality}
	Duality induces a grading-transposing isomorphism of bicomplexes:
	\[
	\begin{tikzcd}[ampersand replacement = \&, row sep = 1em, column sep = 3.5em]
	\left(\MC_{\bullet,\bullet},\partial_{\del},\partial_{\clp}\right) \arrow[r,leftrightarrow,"\sim"] \& \left(\MC_{\bullet,\bullet},\partial_{\lp},\partial_{\con}\right) \\
	\left[\sM,\eta\right] \arrow[r, mapsto] \& \left[\sM^*,\eta\right].
	\end{tikzcd}
	\]
\end{proposition}

\begin{proof}
	As discussed above, the fact that $(\sM^*)^*=\sM$ implies this morphism is an isomorphism of vector spaces. Since duality swaps rank and nullity, duality sends $(\MC_{k,r})$ isomorphically onto $\MC_{r,k}$. That the differentials are compatible follows from the fact that: (i) An element $x\in E(\sM)$ is a loop in $\sM$ if and only if it is a coloop in $\sM^*$ and (ii) $(\sM \setminus x)^*=\sM^*/x$. 
\end{proof}

\begin{remark}
In fact, every horizontal differential anti-commutes with every vertical differential. Equivalently, besides the deletion and contraction bicomplexes, the diagonal pairs
\[
(\MC_{\bullet,\bullet},\partial_{\del},\partial_{\con})
\qquad\text{and}\qquad
(\MC_{\bullet,\bullet},\partial_{\lp},\partial_{\clp})
\]
also define bicomplexes. This follows from the same sign-reversing pairing argument used above, now applied to the sums \(\partial_{\del}+\partial_{\con}\) and \(\partial_{\lp}+\partial_{\clp}\). Moreover, matroid duality exchanges \(\partial_{\del}\) with \(\partial_{\con}\) and \(\partial_{\lp}\) with \(\partial_{\clp}\), so each diagonal bicomplex is preserved by duality after transposing the bigrading.
\end{remark}

%%%%%%%%%%%%%%%%%%%%%%%%%%%%%%%%%%%%%%%%%%%%%%%%%%%
\subsection{Combinatorial Sub-bicomplexes}
%%%%%%%%%%%%%%%%%%%%%%%%%%%%%%%%%%%%%%%%%%%%%%%%%%%

Many natural classes of matroids, such as binary, regular, and graphic matroids, give rise to sub-bicomplexes of the complexes constructed above. Let $\bP$ be a matroid property that is stable under isomorphism, so that it is well-defined on classes $[\sM,\eta]\in\MC$. Define the \emph{$\bP$-subspace} $\MC^{\bP}$ to be the span of all classes with property $\bP$. We can give $\MC^{\bP}$ the nullity-rank bigrading to define $\MC^{\bP}_{\bullet,\bullet}$.

Given the centrality of duality, we define $\bP^{*}$ by saying $\sM$ has $\bP^*$ if and only if $\sM^*$ has $\bP$. For example, if $\bP$ is the property of being loopless, being simple, or being graphic the corresponding $\bP^*$ properties are being coloopless, cosimple, or cographic. The isomorphism in Proposition~\ref{prop:duality} restricts to give an isomorphism between $\MC^{\bP}$ and $\MC^{\bP^{*}}$. Once again this isomorphism swaps the bigrading. Some matroid properties, such as being regular or being representable over a field $\FF$, are stable under duality, in which case $\MC^{\bP} = \MC^{\bP^*}$.

Recall that $\bP$ is \emph{deletion closed} (resp.\ \emph{contraction closed}) if $\sM$ having $\bP$ implies $\sM\setminus x$ (resp.\ $\sM/x$) has $\bP$ for all $x\in E(\sM)$. Note that $\bP$ is deletion closed if and only if $\bP^*$ is contraction closed. As expected, deletion closed properties give subcomplexes of the deletion bicomplex and contraction closed properties give subcomplexes of the contraction bicomplex. 
 
\begin{proposition}\label{lem:property-p-subcomplex}
    Let $\bP$ be a property of matroids stable under isomorphism.
    \begin{enumerate}
        \item If $\bP$ is deletion closed, then $(\MC^{\bP}_{\bullet,\bullet}, \partial_{\del}, \partial_{\clp})$ is a sub-bicomplex of $(\MC_{\bullet,\bullet}, \partial_{\del}, \partial_{\clp})$.
        \item If $\bP$ is contraction closed, then $(\MC^{\bP}_{\bullet,\bullet}, \partial_{\lp}, \partial_{\con})$ is a sub-bicomplex of $(\MC_{\bullet,\bullet}, \partial_{\lp}, \partial_{\con})$.
        \item If $\bP$ is deletion closed, then duality gives a grading-transposing isomorphism between $(\MC^{\bP}_{\bullet,\bullet}, \partial_{\del}, \partial_{\clp})$ and $(\MC^{\bP^{*}}_{\bullet,\bullet}, \partial_{\lp}, \partial_{\con})$.
    \end{enumerate}
\end{proposition}

\begin{proof}
	This follows from previous constructions and the deletion/contraction conditions. 
\end{proof}

\begin{remark}
A property \(\bP\) is \emph{minor closed} if it is both deletion closed and contraction closed. Equivalently, \(\bP\) is minor closed if every minor of a matroid with property \(\bP\) again has property \(\bP\). Thus, if \(\bP\) is minor closed and stable under isomorphism, then both $(\MC^{\bP}_{\bullet,\bullet}, \partial_{\del}, \partial_{\clp})$ and $(\MC^{\bP}_{\bullet,\bullet}, \partial_{\lp}, \partial_{\con})$ are well-defined sub-bicomplexes. 

In general, being minor closed does \emph{not} imply $\bP=\bP^*$. For example, being graphic and being cographic are both minor closed, but duality interchanges these two properties rather than preserving either one. By contrast, being binary, being regular, and being representable over a field \(\FF\) are minor closed and stable under duality.  When $\bP$ is minor closed but not self-dual, Proposition~\ref{lem:property-p-subcomplex} yields two pairs of grading-transposing isomorphic bicomplexes:
\[
(\MC^{\bP}_{\bullet,\bullet}, \partial_{\del}, \partial_{\clp})
\cong
(\MC^{\bP^*}_{\bullet,\bullet}, \partial_{\lp}, \partial_{\con}) \quad \quad \text{and} \quad \quad (\MC^{\bP}_{\bullet,\bullet}, \partial_{\lp}, \partial_{\con}) \cong (\MC^{\bP^*}_{\bullet,\bullet}, \partial_{\del}, \partial_{\clp}).
\]
When $\bP$ is self-dual these two pairs coincide, so there is essentially a single bicomplex up to isomorphism (as with regular matroids or $\FF$-representable matroids). Otherwise, as with graphic and cographic matroids, the deletion and contraction variants give genuinely distinct complexes coming in dual pairs.
\end{remark}

\begin{center}
\begin{tabular}{ c | c | c | c | c }
    $\MC^{\bP}_{\bullet,\bullet}$ & $\bP$ & $\sM$ is... & $\partial$ & Name \\ \hline
    $\MC^{\lpl}_{\bullet,\bullet}$ & $\lpl$ & a loopless matroid & $\del$ & the loopless matroid complex \\ \hline
    $\MC^{\simp}_{\bullet,\bullet}$ & $\simp$ & a simple matroid & $\del$ & the simple matroid complex \\ \hline
    $\MC^{\colpl}_{\bullet,\bullet}$ & $\colpl$ & a coloopless matroid & $\con$ & the coloopless matroid complex \\ \hline
    $\MC^{\cosimp}_{\bullet,\bullet}$ & $\cosimp$ & a cosimple matroid & $\con$ & the cosimple matroid complex \\ \hline
    $\MC^{\FF}_{\bullet,\bullet}$ & $\FF$ & representable over the field $\FF$ & $\del$ / $\con$ & the $\FF$-representable matroid complex \\ \hline
    $\MC^{\reg}_{\bullet,\bullet}$ & $\reg$ & is a regular matroid & $\del$ / $\con$ & the regular matroid complex \\ \hline
    $\MC^{\grph}_{\bullet,\bullet}$ & $\grph$ & is a graphic matroid & $\del$ / $\con$ & the graphic matroid complex \\ \hline
    $\MC^{\cogrph}_{\bullet,\bullet}$ & $\cogrph$ & is a cographic matroid & $\del$ / $\con$ & the cographic matroid complex \\ 
\end{tabular}
\end{center}

%%%%%%%%%%%%%%%%%%%%%%%%%%%%%%%%%%%%%%%%%%%%%%%%%%%
\subsection{Total Complexes and Filtrations}\label{subsec:tot-filt}
%%%%%%%%%%%%%%%%%%%%%%%%%%%%%%%%%%%%%%%%%%%%%%%%%%%

We now record the ordinary chain complexes obtained from the two bicomplexes above. There are two natural ways to pass from a bicomplex to singly graded complexes. The first is by totalization. By definition we have
\[
\operatorname{Tot}^{\del}_n(\MC_{\bullet, \bullet})\coloneqq \bigoplus_{k+r=n}\MC_{k,r}
\qquad\text{and}\qquad
\operatorname{Tot}^{\con}_n(\MC_{\bullet,\bullet})\coloneqq \bigoplus_{k+r=n}\MC_{k,r},
\]
with differentials
\[
\partial_{\del}^{\tot}\coloneqq \partial_{\del}+\partial_{\clp}
\qquad\text{and}\qquad
\partial_{\con}^{\tot}\coloneqq \partial_{\lp}+\partial_{\con}.
\]
These are chain complexes since $(\partial_{\del}^{\tot})^2=(\partial_{\con}^{\tot})^2=0$. Since a class in $\MC_{k,r}$ has ground set size $k+r$, the natural total degree is the size of the ground set. Writing $\MC_{\bullet}$ for $\MC$ considered as a $\ZZ$-graded vector space with grading given by ground set size, we may therefore identify the two totalizations with the complexes $(\MC_{\bullet}, \partial_{\del}^{\tot})$ and $(\MC_{\bullet}, \partial_{\con}^{\tot})$. These two complexes are related by matroid duality.

\begin{lemma}\label{lem:tot-comp-tduality}
	Matroid duality induces an isomorphism of complexes:
	\[
	\begin{tikzcd}[ampersand replacement = \&, row sep = 1em, column sep = 3.5em]
	\left(\MC_{\bullet},\partial_{\del}^{\tot}\right) \arrow[r,leftrightarrow,"\sim"] \& \left(\MC_{\bullet},\partial_{\con}^{\tot}\right).
	\end{tikzcd}
	\]
\end{lemma}

\begin{proof}
This is immediate from Proposition~\ref{prop:duality} after passing to total complexes.
\end{proof}

The second construction is to retain only one of the two differentials. Since each of the maps $\partial_{\del}$, $\partial_{\clp}$, $\partial_{\con}$, and $\partial_{\lp}$ lowers the ground set size by one, each determines a chain complex on $\MC_{\bullet}$. Thus, in addition to the total complexes, we obtain the four complexes
\[
(\MC_{\bullet},\partial_{\del}),\quad
(\MC_{\bullet},\partial_{\clp}),\quad
(\MC_{\bullet},\partial_{\con}),\quad
(\MC_{\bullet},\partial_{\lp}).
\]
We will reserve the adjective ``total'' for the differentials $\partial_{\del}^{\tot}$ and $\partial_{\con}^{\tot}$. Each total complex carries two natural increasing filtrations, one by rank and one by nullity. For the deletion total complex the filtrations are given by
\[
F^{\rk}_p\operatorname{Tot}^{\del}_n(\MC_{\bullet,\bullet})\coloneqq \bigoplus_{\substack{k+r=n\\ r\le p}}\MC_{k,r}
\qquad\text{and}\qquad
F^{\nul}_p\operatorname{Tot}^{\del}_n(\MC_{\bullet,\bullet})\coloneqq \bigoplus_{\substack{k+r=n\\ k\le p}}\MC_{k,r},
\]
and these are preserved by $\partial_{\del}^{\tot}$. Given a filtered complex $F_{p}C_{\bullet}$ we write $\gr_p^F C_\bullet$ for the associated graded complex defined by
\[
\gr_p^F C_\bullet\coloneqq F_pC_\bullet/F_{p-1}C_\bullet.
\]
In the case of the filtrations on the deletion total complex we have
\[
\gr^{\rk}_p\operatorname{Tot}^{\del}_n(\MC_{\bullet,\bullet})\cong \MC_{n-p,p}
\qquad\text{and}\qquad
\gr^{\nul}_p\operatorname{Tot}^{\del}_n(\MC_{\bullet,\bullet})\cong \MC_{p,n-p},
\]
with induced differentials $\partial_{\del}$ and $\partial_{\clp}$ respectively. Thus, after the evident reindexing of total degree, the associated graded pieces are the fixed-rank complexes $(\MC_{\bullet,r},\partial_{\del})$ and the fixed-nullity complexes $(\MC_{k,\bullet},\partial_{\clp})$.

The situation for the contraction total complex is analogous. Again there are two natural increasing filtrations, now given by
\[
F^{\rk}_p\operatorname{Tot}^{\con}_n(\MC_{\bullet,\bullet})\coloneqq \bigoplus_{\substack{k+r=n\\ r\le p}}\MC_{k,r}
\qquad\text{and}\qquad
F^{\nul}_p\operatorname{Tot}^{\con}_n(\MC_{\bullet,\bullet})\coloneqq \bigoplus_{\substack{k+r=n\\ k\le p}}\MC_{k,r},
\]
which are preserved by $\partial_{\con}^{\tot}$. The associated graded complexes are then
\[
\gr^{\rk}_p\operatorname{Tot}^{\con}_n(\MC_{\bullet,\bullet})\cong \MC_{n-p,p}
\qquad\text{and}\qquad
\gr^{\nul}_p\operatorname{Tot}^{\con}_n(\MC_{\bullet,\bullet})\cong \MC_{p,n-p},
\]
with induced differentials $\partial_{\lp}$ and $\partial_{\con}$ respectively. These are the fixed-rank complexes $(\MC_{\bullet,r},\partial_{\lp})$ and fixed-nullity complexes $(\MC_{k,\bullet},\partial_{\con})$. Since the bicomplexes are supported in the first quadrant, all of these filtrations are bounded below and exhaustive. Lastly, these constructions extend to the combinatorial sub-bicomplexes from the previous section.

\begin{proposition}\label{prop:property-totalizations}
Let \(\bP\) be a matroid property stable under isomorphism.
\begin{enumerate}
    \item If \(\bP\) is deletion closed, then all constructions of this subsection yield \(\bP\)-subcomplexes
    \[
 	(\MC^{\bP}_{\bullet}, \partial_{\del}^{\tot}) \subset (\MC_{\bullet}, \partial_{\del}^{\tot}) \qquad
	(\MC^{\bP}_{\bullet}, \partial_{\del}) \subset (\MC_{\bullet}, \partial_{\del}) \qquad
	(\MC^{\bP}_{\bullet}, \partial_{\clp}) \subset (\MC_{\bullet}, \partial_{\clp})
    \]
    together with the induced filtrations, associated graded pieces, and decompositions.

    \item If \(\bP\) is contraction closed, then all constructions of this subsection yield \(\bP\)-subcomplexes
    \[
 	(\MC^{\bP}_{\bullet}, \partial_{\con}^{\tot}) \subset (\MC_{\bullet}, \partial_{\con}^{\tot}) \qquad
	(\MC^{\bP}_{\bullet}, \partial_{\con}) \subset (\MC_{\bullet}, \partial_{\con}) \qquad
	(\MC^{\bP}_{\bullet}, \partial_{\lp}) \subset (\MC_{\bullet}, \partial_{\lp})
    \]
    together with the induced filtrations, associated graded pieces, and decompositions.

    \item Matroid duality identifies each deletion \(\bP\)-subcomplex in {\rm(1)} with the corresponding contraction \(\bP^*\)-subcomplex in {\rm(2)}. More precisely, under duality the rank and nullity filtrations are interchanged.
\end{enumerate}
\end{proposition}

\begin{proof}
	This follows immediately from Proposition~\ref{lem:property-p-subcomplex} together with the definitions of the total complexes, filtrations, associated graded complexes, and direct sum decompositions.
\end{proof}

%%%%%%%%%%%%%%%%%%%%%%%%%%%%%%%%%%%%%%%%%%%%%%%%%%%
\section{Comparison with Earlier Matroid Complexes}\label{sec:comparison}
%%%%%%%%%%%%%%%%%%%%%%%%%%%%%%%%%%%%%%%%%%%%%%%%%%%

The constructions of Section~\ref{sec:matroid-complex} recover the two matroid complexes that appear most directly in the existing literature. The first is the matroid homology complex of Alekseyevskaya, Borovik, Gelfand, and White \cite{ABGW00}. The second is the regular matroid complex from Section~5.2 of \cite{BBCMMW24}. The point is that the former is obtained by keeping only one of the deletion or contraction differentials, while the latter is obtained by restricting the total deletion complex to simple regular matroids and then bounding the rank.

\subsection{The \cite{ABGW00} Complex} 

For the comparison with \cite{ABGW00}, if $E=\{e_1<\cdots<e_n\}$ is a total ordering on the ground set of a matroid $\sM$ write $\omega_E \coloneqq e_1\wedge \cdots \wedge e_n \in \bigwedge^n \ZZ\langle E\rangle$. The degree-$n$ group in \cite{ABGW00} is spanned by ordered matroids $(\sM,E,\le)$, modulo the rule that an isomorphism contributes the sign of the induced permutation of $E$. This is exactly our orientation relation, so for $n\geq 1$ the assignment
\[
(\sM,E,\le)\longmapsto (-1)^n[\sM,\omega_E]
\]
identifies their generators with ours. The only differences in conventions are the following: (i) in degree $0$: \cite{ABGW00} set their chain complex equal to $0$, whereas we retain the class $[\varnothing,1]$ and (ii) they work over $\ZZ$ while we work over $\QQ$; so one should tensor their complex by $\QQ$ for comparison. Moreover, $\iota_{e_i}(\omega_E)=(-1)^{i-1}\omega_{E\setminus\{e_i\}}$, so the factor $(-1)^n$ converts the signs $(-1)^i$ used in \cite{ABGW00} into ours. Thus their deletion and contraction differentials are precisely $\partial_{\del}$ and $\partial_{\con}$. Equivalently, after tensoring by $\QQ$ the complexes of \cite{ABGW00} are the positive-degree truncations of $(\MC_{\bullet},\partial_{\del})$ and $(\MC_{\bullet},\partial_{\con})$, or, in the language of Section~\ref{subsec:tot-filt}, the direct sums of the fixed-rank and fixed-nullity pieces.

\subsection{The \cite{BBCMMW24} Complex} 

Turning to the regular matroid complex defined in \cite{BBCMMW24}, fix an integer $g\geq 0$. If $\sM$ is a simple regular matroid of rank at most $g$, choose a totally unimodular representation $A=[v_e]_{e\in E(\sM)}$ where $v_{e}$ is a column vector. Set $q_e\coloneqq v_ev_e^t$. The associated matroidal cone is $\sigma(\sM)=\RR_{\geq 0}\langle q_e \mid e\in E(\sM)\rangle$ viewed as a cone in the space of all $g\times g$ symmetric real matrices. As $\sM$ is simple, the rays of $\sigma(\sM)$ are distinct, and because matroidal cones are simplicial, every codimension-one face is obtained by omitting exactly one ray. Hence the facets of $\sigma(\sM)$ are precisely the cones $\sigma(\sM\setminus e)$. If $E(\sM)=\{e_1<\cdots<e_n\}$, choose the reference orientation on $\sigma(\sM)$ given by the ordered ray basis $q_{e_1}\wedge \cdots \wedge q_{e_n}$. With this choice, the boundary in the simplicial cone complex is the alternating sum over the facets obtained by deleting one ray, so it agrees with the total deletion differential.

It remains to compare the alternating conditions. A cone automorphism of $\sigma(\sM)$ permutes its rays, hence sends each $q_e$ to some $q_{e'}$. Since $hq_eh^t$ is rank one, this implies $hv_e=\pm v_{e'}$, so every cone automorphism induces a matroid automorphism of $\sM$. Conversely, if $\phi\in \Aut(\sM)$, then permuting the columns of a totally unimodular representation by $\phi$ gives another totally unimodular representation of $\sM$; by the well-definedness of Construction~2.12 of \cite{BBCMMW24}, this permutation is realized by an element of $\GL_g(\ZZ)$ stabilizing $\sigma(\sM)$. Thus $\sigma(\sM)$ is alternating if and only if $[\sM,\eta]$ is nonzero in $\MC$. This proves the following.

\begin{proposition}
	Let $R^{(g)}_{\bullet}$ denote the regular matroid complex \cite{BBCMMW24}. There is an isomorphism of complexes:
	\[
	\begin{tikzcd}[row sep = .5em, column sep = 3.5em]
	R^{(g)} \arrow[r, "\sim"] & F_{g}^{\rk} \Tot_{\bullet} \left( \MC^{\reg,\simp}_{\bullet, \bullet}, \partial_{\del}, \partial_{\clp}\right)\\
	\left[\sigma(\sM)\right] \arrow[r,mapsto] & \left[\sM, \eta\right]
	\end{tikzcd}.
	\] 
\end{proposition}

A consequence of the above proposition is that passing to the top associated graded piece for the rank filtration recovers the fixed-rank simple regular complex
\[
\left(\MC^{\reg}_{\bullet,g}\cap \MC^{\simp}_{\bullet,g},\partial_{\del}\right).
\]

\begin{remark}
The same description identifies the coloop complex $C^{(g)}_{\bullet}$ of \S5.2 of \cite{BBCMMW24} with the subcomplex spanned by simple regular matroids of rank $<g$, together with those of rank $g$ having a coloop.
\end{remark}

%%%%%%%%%%%%%%%%%%%%%%%%%%%%%%%%%%%%%%%%%%%%%%%%%%%
\section{Algebraic Structures on Matroid Complexes}\label{sec:algebraic-structure}
%%%%%%%%%%%%%%%%%%%%%%%%%%%%%%%%%%%%%%%%%%%%%%%%%%%

In this section we use direct sums, restrictions, and contractions to endow the matroid complexes with additional algebraic structure, proving the following. 

\begin{theorem}\label{thm:hopf-algebra}
    The direct sum product $\star$ and the coproduct $\Delta$ make $(\MC_{\bullet},\star,[\varnothing,1],\Delta,\epsilon)$ into a connected graded Hopf algebra such that:
    \begin{enumerate}
    	\item  $\partial_{\del}$, $\partial_{\clp}$, and $\partial^{\tot}_{\del}$ are right graded coderivations for $\Delta$, 
	\item $\partial_{\con}$, $\partial_{\lp}$, and $\partial_{\con}^{\tot}$ are left graded coderivations for $\Delta$, and 
	\item all of the above differentials are graded derivations for $\star$. 
    \end{enumerate}
     Moreover, the rank and nullity filtrations on $\MC_{\bullet}$ are Hopf algebra filtrations, and the corresponding associated graded objects are connected graded Hopf algebras.
\end{theorem}

Before giving the formal definitions, we record the picture to keep in mind. There are two basic ways to build algebraic structure from matroids: direct sums and restriction-contraction. The direct sum takes two matroids and produces a larger one, so it naturally gives a product $\star:\MC\otimes \MC\to \MC$. For instance, $[U_{1,1},1]\star [U_{1,2},2\wedge 3] = [U_{1,1}\oplus U_{1,2},1\wedge 2\wedge 3]$. By contrast, restriction and contraction take a single matroid \(\sM\) together with a subset \(S\subseteq E(\sM)\) and produce an ordered pair $(\sM|S,\sM/S)$ and so summing over all subsets gives a coproduct $\Delta:\MC\to \MC\otimes \MC$.

The asymmetry between these two constructions should be kept in mind from the outset. The direct sum preserves rank and nullity separately, so $\star$ is naturally bihomogeneous. The coproduct $\Delta$ only preserves rank and nullity additively across the two tensor factors. Thus, while $\Delta$ is compatible with the bigrading, the most natural connected grading for the Hopf algebra structure is the grading by ground set size.

Finally, the differentials interact with these constructions in two distinct ways. For the product $\star$, a single element of $\sM\oplus \sQ$ belongs to one summand or the other, so each differential satisfies the Leibniz rule. For the coproduct $\Delta$, the ordered pair $(\sM|S,\sM/S)$ remembers which elements lie in $S$ and which lie in its complement. This forces an asymmetry: deletion-type differentials act on the contraction factor, while contraction-type differentials act on the restriction factor. This means the differentials are not two-sided coderivations in general. Instead, deletion-type differentials are right coderivations and contraction-type differentials are left coderivations.

As a consequence of Theorem~\ref{thm:hopf-algebra} we will see that essentially any matroid complex that contains either $U_{0,1}$ or $U_{1,1}$ and respects the product structure $\star$ will be acyclic. 

We now make this precise in stages. We first treat the direct sum product and show that each of our four differentials satisfies a Leibniz rule with respect to $\star$. We then construct the restriction-contraction coproduct and prove the corresponding one-sided coderivation identities. Once these ingredients are in place, Theorem~\ref{thm:hopf-algebra} is largely formal, and the rest of the section extracts its two main consequences: broad acyclicity results and a reduction of the algebra structure to connected matroids.

%%%%%%%%%%%%%%%%%%%%%%%%%%%%%%%%%%%%%%%%%%%%%%%%%%%
\subsection{Direct Sum Product}
%%%%%%%%%%%%%%%%%%%%%%%%%%%%%%%%%%%%%%%%%%%%%%%%%%%

We begin by defining a product on $\MC$ using direct sums of matroids. If $\sM$ and $\sQ$ are matroids, the ground set of $\sM\oplus \sQ$ is $E(\sM)\sqcup E(\sQ)$; so there is a canonical isomorphism 
\begin{equation}\label{eqn:orientation-wedge}
\begin{tikzcd}[column sep =3.25em, row sep = .5em]
    \displaystyle \bigwedge^{|E(\sM)|}\ZZ\left\langle E(\sM)\right\rangle \otimes \bigwedge^{|E(\sQ)|}\ZZ\left\langle E(\sQ)\right\rangle \arrow[r, "\sim"] & \bigwedge^{|E(\sM\oplus\sQ)|}\ZZ\left\langle E(\sM\oplus \sQ)\right\rangle.
\end{tikzcd}
\end{equation}
Thus, if $\eta$ and $\omega$ are orientations on $\sM$ and $\sQ$ respectively then $\eta\wedge \omega$ is an orientation on $\sM\oplus\sQ$. Using this, we define a bilinear map  $\widetilde{\star}$ on $\widetilde{\cM}$ on generators by
    \[
    \begin{tikzcd}[column sep =3.25em, row sep = .5em]
        \widetilde{\cM} \otimes\widetilde{\cM} \arrow[r,"\widetilde{\star}"] & \widetilde{\cM}\\ 
        \left(\sM,\eta\right)\otimes\left(\sQ,\omega\right) \arrow[r,mapsto] & \left(\sM\oplus\sQ,\eta\wedge \omega\right).
    \end{tikzcd}
    \]
and then extending linearly. We call this the \emph{direct sum product}. As expected, this descends to a well-defined product on $\MC$, which we will also call the direct sum product. 

\begin{lemma}\label{lem:direct-sum-product}
    The map $\widetilde{\star}$ descends to a bilinear map
    \[
    \begin{tikzcd}[column sep = 3.25em, row sep = .5em]
        \MC\otimes \MC \arrow[r,"\star"] & \MC,
    \end{tikzcd}
    \]
    and the induced product on $\MC$ is associative and unital with $[\varnothing,1]$ as its two-sided unit.
\end{lemma}

\begin{proof}
    To show that $\widetilde{\star}$ descends to a well-defined map on $\MC$ it is enough to show that $\widetilde{\star}(\widetilde{\cR} \otimes \widetilde{\cM} + \widetilde{\cM} \otimes \widetilde{\cR}) \subset \widetilde{\cR}$. We do this directly on generators of $\widetilde{\cR}$. Beginning with the orientation-reversing relation and using linearity of $\widetilde{\star}$ we have
    \begin{align*}
    \widetilde{\star}\left(\left( \left(\sM, \eta\right) + \left(\sM, -\eta\right)\right)\otimes \left(\sQ,\omega\right)\right) &=  \widetilde{\star}\left( \left(\sM, \eta\right) \otimes \left(\sQ, \omega\right)\right) + \widetilde{\star}\left( \left(\sM, -\eta\right) \otimes \left(\sQ, \omega\right)\right) \\
    &=\left(\sM \oplus \sQ, \eta\wedge\omega\right) + \left(\sM\oplus \sQ, (-\eta)\wedge \omega\right) = \left(\sM \oplus \sQ, \eta\wedge\omega\right) + \left(\sM\oplus \sQ, -(\eta\wedge \omega)\right)
    \end{align*}
    which is clearly in $\widetilde{\cR}$. Note the last equality above follows from the linearity of \eqref{eqn:orientation-wedge}. The isomorphism relation is checked analogously, using that if $\psi: \sM \to \sM'$ is an isomorphism of matroids then $(\psi \oplus \id):\sM\oplus \sQ \to \sM'\oplus \sQ$ is an isomorphism and $(\psi\oplus\id)_{*}(\eta\wedge\omega) = (\psi_*(\eta))\wedge \omega$.
    
    Associativity follows immediately from the definition and the facts that $(\sM \oplus \sQ)\oplus \sP$ is equal to $\sM\oplus (\sQ \oplus \sP)$ and $(\eta \wedge \omega) \wedge \zeta = \eta \wedge (\omega \wedge \zeta)$. Lastly one checks that $[\varnothing,1]$ is a two-sided unit via a direct computation using that $\sM\oplus\varnothing\cong \sM$ and $\wedge^{|E(\sM)|}\ZZ\langle E(\sM)\rangle \wedge 1 \cong \wedge^{|E(\sM)|}\ZZ\langle E(\sM)\rangle$.
\end{proof}

The direct sum product on $\MC$ is homogeneous for all of the gradings of interest. Moreover, while it is not commutative, it is graded-commutative with respect to the total grading on $\MC_{\bullet}$ given by ground set size. 

\begin{lemma}\label{lem:direct-sum-gradings}
\begin{enumerate}
	\item The morphism $\star$ is homogeneous of degree $(0,0)$ on $\MC_{\bullet,\bullet}$.
	\item The morphism $\star$ is homogeneous of degree $0$ on $\MC_{\bullet}$.
	\item  The direct sum product is graded-commutative on $\MC_{\bullet}$:
	\[
	\left[\sM,\eta\right]\star\left[\sQ,\omega\right] = (-1)^{|E(\sM)|\,|E(\sQ)|}\left[\sQ,\omega\right]\star \left[\sM,\eta\right].
	\] 
\end{enumerate}
\end{lemma}

\begin{proof}
	Part (2) follows from (1) by formal properties of graded maps. Condition (1) is equivalent to showing that for all $k,r,k',r'\in \ZZ$ one has
\[
    \begin{tikzcd}[row sep = .5em, column sep = 3.25em]
    \MC_{k,r}\otimes \MC_{k',r'} \arrow[r,"\star"] & \MC_{k+k',r+r'}
     \end{tikzcd}
    \]
     This follows from the fact that  $\rk(\sM\oplus \sQ)=\rk(\sM)+\rk(\sQ)$ and $\nul(\sM\oplus \sQ)=\nul(\sM)+\nul(\sQ)$.
     
     Turning to part (3) it suffices to check the claim on generators.  Using the definition of $\star$ and that matroid direct sum is symmetric the left-hand side can be written as
\[
\left[\sM,\eta\right]\star\left[\sQ,\omega\right] = \left[\sM\oplus\sQ,\eta\wedge\omega\right] =\left[\sQ\oplus\sM,\eta\wedge\omega\right] 
\]
       It is thus enough to prove that $\eta \wedge \omega = (-1)^{|E(\sM)|\,|E(\sQ)|} \omega \wedge \eta$, which follows from the usual sign rule in the exterior algebra. 
       \end{proof}

The four differentials defined in the previous section all satisfy the Leibniz rule with respect to $\star$. The uniform behavior here 
contrasts with the asymmetric behavior we will see in \S\ref{subsec:coproduct} for the coproduct. This is because deletion and contraction of a single element both commute with taking a direct sum with a disjoint matroid: for $s \in E(\sM)$ one has  $(\sM \oplus \sQ)\setminus s = (\sM\setminus s)\oplus \sQ$ and  $(\sM \oplus \sQ)/s = (\sM/s) \oplus \sQ$, and analogously for elements of $E(\sQ)$.

\begin{lemma}\label{lem:direct-sum-derivation}
    Let $\partial$ denote any of $\partial_{\del}$, $\partial_{\clp}$, $\partial_{\con}$, or $\partial_{\lp}$. If $a,b\in \MC_{\bullet}$ are homogeneous, then
    \[
    \partial(a\star b)=\partial(a)\star b+(-1)^{|a|}a\star \partial(b).
    \]
\end{lemma}

\begin{proof}
    Let us begin with $\partial_{\del}$. It is enough to check the claim on generators, and we proceed by computing the left-hand side directly
    \begin{align*}
        \partial_{\del}\left( \left[\sM,\eta\right]\star\left[\sQ,\omega\right]\right)&=\partial_{\del}\left(\left[\sM\oplus\sQ,\eta\wedge\omega\right]\right)=\sum_{\substack{x\in E(\sM\oplus\sQ)\\ \text{$x$ not a coloop}}}\left[(\sM\oplus\sQ)\setminus x, (\eta\wedge\omega)\setminus x\right]\\
        &=\sum_{\substack{s \in E(\sM)\\ \text{not a coloop}}} \left[(\sM\oplus \sQ)\setminus s, (\eta\wedge\omega)\setminus s\right]+
        \sum_{\substack{t \in E(\sQ)\\ \text{not a coloop}}} \left[(\sM\oplus\sQ)\setminus t, (\eta\wedge\omega)\setminus t\right]
    \end{align*}
    Deletion of a single element and direct sum commute, meaning $(\sM\oplus \sQ)\setminus s =(\sM\setminus s)\oplus\sQ$ and $(\sM\oplus \sQ)\setminus t =\sM\oplus(\sQ\setminus t)$. We must thus understand how deletion behaves with respect to wedge product of orientations. Recall from Definition~\ref{def:orientation} that we defined the deletion $(\eta\wedge\omega)\setminus s$ to be the interior product $\iota_{s}(\eta\wedge\omega)$ where $\iota_s$ is the interior product given by $\iota_s(s\wedge\alpha)=\alpha$. Abusing notation slightly, we may actually extend $\iota_s$ to be a degree $-1$ graded derivation on the exterior algebra:
    \[
    \begin{tikzcd}[column sep = 3em]
    \iota_{s}:\displaystyle \bigwedge^{\bullet}\ZZ\left\langle E(\sM)\sqcup E(\sQ) \right\rangle \arrow[r]& \displaystyle\bigwedge^{\bullet}\ZZ\left\langle E(\sM)\sqcup E(\sQ) \right\rangle
    \end{tikzcd}
    \]
    satisfies the usual relation $\iota_{s}(\alpha\wedge \beta)=\iota_{s}(\alpha)\wedge\beta+(-1)^{|\alpha|}\alpha\wedge\iota_{s}(\beta)$. Applying this relation, and using that $\iota_{s}(\omega)=0$ and $s\in E(\sM)$ and $\omega$ is disjoint from the span of $E(\sM)$, we see that $\iota_s(\eta\wedge \omega)=\iota_s(\eta)\wedge \omega$. This means that $(\eta\wedge \omega)\setminus s = (\eta\setminus s)\wedge \omega$. The analogous argument shows that $(\eta\wedge \omega)\setminus t=(-1)^{|\eta|}\eta \wedge(\omega \setminus t)$. Combining all of this gives
    \begin{align*}
        \partial_{\del}\left( \left[\sM,\eta\right]\star\left[\sQ,\omega\right]\right)&=\sum_{\substack{s \in E(\sM)\\ \text{not a coloop}}} \left[(\sM\setminus s)\oplus \sQ, (\eta\setminus s)\wedge\omega\right]+
        \sum_{\substack{t \in E(\sQ)\\ \text{not a coloop}}} (-1)^{|E(\sM)|}\left[\sM\oplus(\sQ\setminus t), \eta\wedge(\omega\setminus t)\right]\\
        &=\sum_{\substack{s \in E(\sM)\\ \text{not a coloop}}} \left[\sM\setminus s, \eta\setminus s\right]\star\left[\sQ,\omega\right]+
        \sum_{\substack{t \in E(\sQ)\\ \text{not a coloop}}} (-1)^{|E(\sM)|}\left[\sM,\eta\right]\star\left[\sQ\setminus t, \omega\setminus t\right]\\
    &=\partial_{\del}\left(\left[\sM, \eta\right]\right) \star \left[\sQ,\omega \right]+(-1)^{|E(\sM)|}\left[\sM,\eta\right]\star\partial_{\del}\left(\left[\sQ,\omega\right]\right)
    \end{align*}
	This proves the formula for $\partial_{\del}$. The other cases are largely analogous. The only differences are replacing ``not a coloop'' by ``a coloop'', ``not a loop'', and ``a loop'' as appropriate and using that if $s \in E(\sM)$ and $t \in E(\sQ)$ then $(\sM\oplus\sQ)/s = (\sM/s)\oplus \sQ$ and $(\sM\oplus\sQ)/t = \sM\oplus (\sQ/t)$. 
\end{proof}

%%%%%%%%%%%%%%%%%%%%%%%%%%%%%%%%%%%%%%%%%%%%%%%%%%%
\subsection{The Restriction-Contraction Coproduct}\label{subsec:coproduct}
%%%%%%%%%%%%%%%%%%%%%%%%%%%%%%%%%%%%%%%%%%%%%%%%%%%

We now turn to defining a coproduct on $\MC$ using the restriction and contraction operations on matroids. Unlike the direct sum product, this coproduct cannot be reasonably defined on $\widetilde{\cM}$ itself. The issue is that in order to define the coproduct we must understand how an orientation $\eta$ on $\sM$ induces orientations on the restriction $\sM|S$ and contraction $\sM/S$ for all subsets $S \subset E(\sM)$. There is no canonical way to make such a choice. However, the choice of an induced orientation on both $\sM|S$ and $\sM/S$ is canonical up to sign. This sign ambiguity disappears in $\MC\otimes\MC$, since reversing both induced orientations multiplies both factors by $-1$, and so the signs cancel. We now give the definition of the coproduct before unpacking the necessary operations to make sense of it or check it is well-defined. 

\begin{definition}\label{def:coproduct}
    The \emph{$\Delta$ coproduct} on $\MC$ is the linear map
    \[
    \begin{tikzcd}[column sep = 3.25em, row sep = .75em]
        \MC \arrow[r,"\Delta"] & \MC \otimes \MC\\
        {[\sM,\eta]} \arrow[r,mapsto] & \displaystyle \sum_{S\subseteq E(\sM)} [\sM|S,\eta|S]\otimes [\sM/S,\eta/S].
    \end{tikzcd}
    \]
\end{definition}

\begin{remark}
The coproduct $\Delta$ is the oriented analogue of the restriction-contraction coproduct appearing in Schmitt's Hopf algebra of matroids \cite{schmitt94}. 
In Schmitt's construction one works with the free $\QQ$-vector space on isomorphism classes of matroids (with no orientation data), and the 
coproduct $\Delta([\sM]) = \sum_{S \subseteq E(\sM)} [\sM|S]\otimes[\sM/S]$ is immediately well-defined. The novelty here is that our orientation data introduces signs into the formula, which is precisely what makes $\Delta$ compatible with the chain complex differentials.
\end{remark}

Making sense of this definition requires us to understand how the choice of an orientation $\eta$ on a matroid $\sM$ induces orientations $\eta|S$ and $\eta/S$ on $\sM|S$ and $\sM/S$ for arbitrary subsets $S\subset E(\sM)$. Recall that $\sM|S$, the restriction of $\sM$ to $S$, is equal to the deletion $\sM\setminus(E(\sM) \setminus S)$. The use of restriction rather than deletion is simply for ease of notation. 

Fix a matroid $\sM=(E,\cB)$ and an orientation $\eta \in \left(\bigwedge^{|E|}\ZZ\langle E\rangle\right)^{\times}.$ For any subset $S \subset E$ we have a direct sum decomposition $\ZZ\langle E\rangle \cong \ZZ\langle S \rangle \oplus \ZZ\langle E\setminus S\rangle$, which gives canonical isomorphisms of rank-one free modules:
\[
\begin{tikzcd}[row sep =.75em, column sep = 3.5em]
    \displaystyle \bigwedge^{|S|}\ZZ\langle S\rangle  \otimes \bigwedge^{|E\setminus S|}\ZZ\langle E\setminus S\rangle \arrow[r,"\mu_{S}"] & \bigwedge^{|E|}\ZZ\langle E\rangle.
\end{tikzcd}
\]
In particular, $\mu_{S}^{-1}(\eta)$ is a simple tensor of the form $(\eta|S)\otimes(\eta/S)$ for some $\eta|S \in \left(\bigwedge^{|S|}\ZZ\langle S\rangle\right)^{\times}$ and $\eta/S \in \left(\bigwedge^{|E\setminus S|}\ZZ\langle E\setminus S\rangle\right)^{\times}$. Note that $\eta|S$ and $\eta/S$ are not necessarily unique if $|S|>1$, however, they are unique up to the choice of simultaneous sign. We choose, once and for all, some representative pair
$(\eta|S,\eta/S)$ with $\mu_{S}(\eta|S\otimes\eta/S)=\eta$ and call these the
\emph{induced restriction orientation} on $\sM|S$ and the \emph{induced
contraction orientation} on $\sM/S$. A few remarks:
The key feature of this construction is that although the pair 
$(\eta|S, \eta/S)$ is not unique --- it is determined only up to a 
simultaneous sign flip $(\eta|S,\eta/S) \mapsto (-\eta|S,-\eta/S)$ 
--- this ambiguity cancels in the tensor product:
\begin{enumerate}
    \item The class $[\sM|S,\eta|S]\otimes[\sM/S,\eta/S]$ in $\MC_{\bullet}\otimes\MC_{\bullet}$ is well-defined and 
    \emph{independent} of the choice of preimage $\mu_S^{-1}(\eta)$, since replacing $(\eta|S,\eta/S)$ by $(-\eta|S,-\eta/S)$  changes both tensor factors by $-1$ and leaves their product unchanged.
    
    \item If $S=\varnothing$ then $\ZZ\langle S\rangle \cong \ZZ$  and we take $\eta|\varnothing \coloneqq 1$, the canonical orientation of the empty matroid, and $\eta/\varnothing \coloneqq \eta$. At the other extreme, if $S=E$ we take $\eta|E \coloneqq \eta$ and $\eta/E \coloneqq 1$.
    
    \item If $|S|=1$, say $S=\{x\}$, then $\sM|S$ is the matroid on $\{x\}$, and it has rank $1$ if $x$ is not a loop and rank $0$ if $x$ is a loop. In this case one convenient choice is $\eta|S \coloneqq x$ and $\eta/S \coloneqq \iota_x(\eta)$, since $\mu_S(x\otimes \iota_x(\eta))=\eta$. Hence $\eta/S$ agrees with the orientation obtained from $\eta$ by interior product with respect to $x$ in Definition~\ref{def:orientation}; equivalently, as an element of $\bigwedge^{|E|-1}\ZZ\langle E\setminus\{x\}\rangle$, it is the same element that defines both $\eta\setminus x$ and $\eta/x$. The factor $\eta|S$ is simply the induced orientation on the singleton restriction $\sM|S$.

\end{enumerate}
It is point~(1) that makes it more natural to define the coproduct directly on $\MC$ rather than on $\widetilde{\cM}$. With this in hand we now turn to verifying that $\Delta$ is well-defined and satisfies the expected properties.

For readers who prefer a concrete model, one may keep the following special case in mind. Suppose $E=\{e_1<\cdots<e_n\}$ and $\eta=e_1\wedge\cdots\wedge e_n$. If $S=\{e_{i_1}<\cdots<e_{i_k}\}$ and $E\setminus S=\{e_{j_1}<\cdots<e_{j_{n-k}}\}$, let $\sigma_S$ be the shuffle sending $(1,\dots,n)$ to $(i_1,\dots,i_k,j_1,\dots,j_{n-k})$. Then one convenient choice is
\[
\eta|S=e_{i_1}\wedge\cdots\wedge e_{i_k},
\qquad
\eta/S=\sgn(\sigma_S)\,e_{j_1}\wedge\cdots\wedge e_{j_{n-k}},
\]
since $\mu_S(\eta|S\otimes\eta/S)=\eta$. In other words, the contraction orientation is obtained by writing the elements of $S$ first and then correcting by the sign of the resulting shuffle. This is exactly the sign bookkeeping that appears later in the proofs.

\begin{example}\label{ex:K4-3}
We now revisit the running example of $[\sM(K_4), \eta]$ in the context of the restriction-contraction coproduct. Since $\sM(K_4)$ has 6 ground set elements, writing out $\Delta([\sM(K_4),\eta])$ in full would have 64 terms, and so we only consider some example terms. For instance, when $S=\{2\}$ from our conventions in Examples~\ref{ex:K4-1} and \ref{ex:K4-2} we have $\eta|\{2\}=2$ and $\eta/\{2\}=-1\wedge 3\wedge 4\wedge 5\wedge 6=\eta_2$. Alternatively, since 
\[
2\wedge\left(-1\wedge 3\wedge 4\wedge 5\wedge 6\right)
=
1\wedge 2\wedge 3\wedge 4\wedge 5\wedge 6
=
\eta,
\]
we could have taken this pair to be $-2$ and $1\wedge 3\wedge 4\wedge 5\wedge 6$. Notice in $\widetilde{\cM}\otimes\widetilde{\cM}$ these different orientations give different classes. However, since both factors of the tensor change sign, the overall sign remains unchanged, giving the same class in \(\MC\otimes \MC\). This is exactly the simultaneous sign ambiguity discussed above. 

The first and last terms of the coproduct are $[\varnothing,1]\otimes[\sM(K_4),\eta]$ and $[\sM(K_4),\eta] \otimes [\varnothing, 1]$ and the remaining $62$ terms are indexed by the nonempty proper subsets of $[6]$. 
\end{example}

Before moving forward there are again two separate points to keep in mind. The well-definedness of $\Delta$ is bookkeeping: the simultaneous sign ambiguity in choosing \((\eta|S,\eta/S)\) cancels in the tensor product, and isomorphisms of matroids carry compatible decompositions of the ground-set orientation. Coassociativity is the conceptual point: applying $\Delta$ twice simply refines the ground set into an ordered partition $T\sqcup U\sqcup V=E(\sM)$, and both iterated coproducts sum over the same ordered partitions. The only non-formal step is to check that the induced orientations agree.

\begin{lemma}\label{lem:coproduct}
    The $\Delta$ coproduct is well-defined, coassociative, and counital with counit $\epsilon:\MC_{\bullet}\to \QQ$  defined by $\epsilon([\varnothing,1])=1$ and zero otherwise.
\end{lemma}

\begin{proof}
We first check that $\Delta$ is well-defined on $\MC$. It suffices to show that it respects the two types of relations generating $\widetilde{\cR}$. For the orientation-reversing relation, fix a matroid $\sM$ with orientation $\eta$. For each subset $S\subseteq E(\sM)$, by the discussion preceding the lemma we may take the induced orientations for $-\eta$ to be $-\eta|S$ and $\eta/S$. Thus
\[
\Delta([\sM,-\eta])=\sum_{S\subseteq E(\sM)} [\sM|S,-\eta|S]\otimes[\sM/S,\eta/S]= -\Delta([\sM,\eta]),
\]
and so $\Delta([\sM,\eta]+[\sM,-\eta])=0$. Turning to the isomorphism relation let $\psi:\sM\to \sM'$ be an isomorphism. For each subset $S\subseteq E(\sM)$ the map $\psi$ induces isomorphisms
\[
\psi_{S}:\sM|S\to \sM'|\psi(S)
\qquad \text{and} \qquad
\psi^{S}:\sM/S\to \sM'/\psi(S),
\]
such that the following diagram commutes
\[
\begin{tikzcd}[column sep = 4em, row sep = 3em]
\displaystyle \bigwedge^{|S|}\ZZ\langle S\rangle \otimes \bigwedge^{|E(\sM)\setminus S|}\ZZ\langle E(\sM)\setminus S\rangle \arrow[r,"\mu_{S}"] \arrow[d,"(\psi_{S})_*\otimes \psi^S_*"] & \displaystyle \bigwedge^{|E(\sM)|}\ZZ\langle E(\sM)\rangle \arrow[d,"\psi_{*}"] \\
\displaystyle \bigwedge^{|S|}\ZZ\langle \psi(S)\rangle \otimes \bigwedge^{|E(\sM')\setminus \psi(S)|}\ZZ\langle E(\sM')\setminus \psi(S)\rangle \arrow[r,"\mu_{\psi(S)}"] & \displaystyle \bigwedge^{|E(\sM')|}\ZZ\langle E(\sM')\rangle
\end{tikzcd}.
\]
From this we see that if $\mu_{S}(\eta|S\otimes \eta/S)=\eta$ then
\[
\mu_{\psi(S)}\left((\psi_{S})_*\left(\eta|S\right)\otimes \psi^{S}_*\left(\eta/S\right)\right)=\psi_{*}(\eta),
\]
which means that, up to a simultaneous sign that does not matter on $\MC\otimes \MC$ we have that $(\psi_{S})_*(\eta|S)=\psi_*(\eta)|\psi(S)$ and $\psi^S_*(\eta/S)=\psi_*(\eta)/\psi(S)$. Thus, for any subset $S \subset E(\sM)$ we have the following equalities in $\MC\otimes \MC$
\begin{align*}
\left[\sM|S, \eta|S\right] \otimes [\sM/S, \eta/S] &= \left[\sM'|\psi(S), (\psi_{S})_{*}(\eta|S)\right] \otimes \left[\sM'/\psi(S), \psi^{S}_{*}(\eta/S)\right] 
\\&= 
\left[\sM'|\psi(S), \psi_*(\eta)|\psi(S)\right] \otimes \left[\sM'/\psi(S), \psi_*(\eta)/\psi(S)\right] 
\end{align*}
from which we get $\Delta\left(\left[\sM, \eta\right]\right) - \Delta\left(\left[\sM', \psi_*(\eta)\right]\right)=0$ implying that $\Delta$ is well-defined.

Turning to coassociativity we must show that $(\Delta \otimes \id)(\Delta([\sM,\eta]))$ is equal to $(\id \otimes \Delta)(\Delta([\sM,\eta]))$. By definition we have that 
\begin{align*}
(\Delta \otimes \id)\left(\Delta\left(\left[\sM,\eta\right]\right)\right) &= \sum_{S \subset E(\sM)} \Delta\left(\left[\sM|S, \eta|S\right]\right) \otimes \left[\sM/S, \eta/S\right] \\
&=
\sum_{S \subset E(\sM)} \sum_{T \subset S} \left[ (\sM|S)|T, (\eta|S)|T \right] \otimes \left[ (\sM|S)/T, (\eta|S)/T\right] \otimes \left[\sM/S, \eta/S\right],
\intertext{which, after setting $U=S \setminus T$ and $V=E(\sM)\setminus S$ becomes}
&= \sum_{T \sqcup U \sqcup V = E(\sM)} \left[\sM|T, \eta|T\right]\otimes \left[(\sM|(T\sqcup U))/T,\; (\eta|(T\sqcup U))/T \right] \otimes \left[\sM/(T\sqcup U), \eta/(T\sqcup U)\right],
\end{align*}
where the last sum is over all ordered partitions of $T \sqcup U \sqcup V=E(\sM)$. Similarly we have that 
\[
(\id \otimes \Delta)\left(\Delta\left(\left[\sM,\eta\right]\right)\right) =  \sum_{T \sqcup U \sqcup V = E(\sM)} \left[\sM|T, \eta|T\right] \otimes \left[(\sM/T)|U, (\eta/T)|U \right] \otimes \left[ (\sM/T)/U, (\eta/T)/U\right].
\]
Fixing a partition $E(\sM)=T\sqcup U\sqcup V$, there is a canonical isomorphism 
\[
\begin{tikzcd}[row sep = .5em, column sep = 3.25em]
\displaystyle \bigwedge^{|T|} \ZZ\langle T\rangle \otimes \bigwedge^{|U|} \ZZ\langle U\rangle \otimes \bigwedge^{|V|} \ZZ\langle V\rangle \arrow[r] & \bigwedge^{|E(\sM)|} \ZZ\langle E(\sM)\rangle
\end{tikzcd},
\]
and so we can choose primitive elements $\eta_{T},\eta_{U},\eta_{V}$, unique up to simultaneous signs whose product is $+1$, so that $\eta$ is the image of $\eta_{T}\otimes \eta_{U} \otimes \eta_{V}$. A diagram chase similar to the one used in the proof of well-definedness shows that
\[
\eta|T=\eta_T,\qquad
(\eta|(T\sqcup U))/T=(\eta/T)|U=\eta_U,\qquad
\eta/(T\sqcup U)=(\eta/T)/U=\eta_V.
\]
Combined with the fact that restriction and contraction satisfy
\[
\left(\sM|(T\sqcup U)\right)/T = (\sM/T)|U
\qquad \text{and} \qquad
\sM/(T\sqcup U) = \left(\sM/T\right)/U,
\]
this shows that the terms corresponding to the partition $E(\sM)=T\sqcup U\sqcup V$ appearing in $(\Delta \otimes \id)(\Delta([\sM,\eta]))$ and in $(\id \otimes \Delta)(\Delta([\sM,\eta]))$ are equal. 

Lastly we check the counit identity, which we do by direct computation:
\[
\left(\epsilon\otimes \id \right)\left(\Delta\left(\left[\sM,\eta\right]\right)\right) = \sum_{S\subset E(\sM)} \epsilon\left(\left[\sM|S,\eta|S \right]\right) \otimes \left[\sM/S,\eta/S\right],
\]
which is equal to $[\sM,\eta]$ as $\epsilon([\sM|S,\eta|S])$ is zero unless $S=\varnothing$, in which case $\sM/S=\sM$. A similar calculation shows that the only term of $\left(\id \otimes \epsilon \right)\left(\Delta\left(\left[\sM,\eta\right]\right)\right)$ that is nonzero occurs when $S=E(\sM)$ in which case $\sM|S=\sM$. 
\end{proof}

\begin{lemma}\label{lem:coproduct-gradings}
\begin{enumerate}
	\item The morphism $\Delta$ is homogeneous of degree $(0,0)$ on $\MC_{\bullet,\bullet}$.
	\item The morphism $\Delta$ is homogeneous of degree $0$ on $\MC_{\bullet}$.
\end{enumerate}
\end{lemma}

\begin{proof}
Part (2) follows from (1) by formal properties of graded morphisms. For (1) it is enough to check that if $[\sM,\eta]\in \MC_{k,r}$, then every summand in $\Delta([\sM,\eta])$ has nullity adding to $k$ and rank adding to $r$. For a subset  $S\subseteq E(\sM)$, by definition, a basis of $\sM/S$ is obtained by removing from a basis of $\sM$ a maximal independent subset of $S$. Combining this with the fact that $E(\sM|S)=S$ and $E(\sM/S)=E(\sM)\setminus S$ gives:
\[
\rk(\sM) = \rk(\sM/S)+\rk(\sM|S) \quad \quad \text{and} \quad \quad \nul(\sM) = \nul(\sM|S)+\nul(\sM/S).
\]
This proves the claim since every term of $\Delta([\sM,\eta])$ is of the form $[\sM|S,\eta|S]\otimes [\sM/S, \eta/S]$. 
\end{proof}

We now turn to exploring the compatibility between the various differentials on $\MC$ and  the coproduct we have just defined. Before doing so, we briefly recall the basic definitions of coderivations since these seem less common.  Going forward we follow the usual Koszul sign convention for how graded maps act on tensor products. Namely if $f:C_{\bullet} \to C_{\bullet}$ is a morphism of $\ZZ$-graded vector spaces of degree $d$ then we define $(\id \otimes f)$ to be the morphism $C_{\bullet}\otimes C_{\bullet} \to C_{\bullet}\otimes C_{\bullet}$ given on homogeneous simple tensors by
\[
\left(\id \otimes f\right)(a\otimes b) = (-1)^{d|a|} a \otimes f(b)
\]
where $|a|$ denotes the degree of $a$. Further, suppose $C_{\bullet}$ is equipped with a coassociative coproduct $\Delta:C_{\bullet}\to C_{\bullet}\otimes C_{\bullet}$ of degree $0$. A homogeneous linear map $D:C_{\bullet}\to C_{\bullet}$ of degree $d$ is called a \emph{graded coderivation} with respect to $\Delta$ if the co-Leibniz identity
\[
\Delta \circ D=(D\otimes \id+\id\otimes D)\circ \Delta
\]
holds. Using the Koszul rule above, this identity may be written explicitly on homogeneous elements $x\in C_{\bullet}$ as
\[
\Delta(D(x))=\sum_{i} \left(D(a_i) \otimes b_i + (-1)^{d |a_i|} a_i \otimes D(b_i) \right)
\]
where $\Delta(x)=\sum_i a_i \otimes b_i$.
More generally, one can consider one-sided versions of this condition. A degree $d$ map $D:C_{\bullet}\to C_{\bullet}$ is called a \emph{left coderivation} if
$\Delta \circ D=(D\otimes \id)\circ \Delta$ and a \emph{right coderivation} if $\Delta \circ D=(\id\otimes D)\circ \Delta$.
Equivalently, on homogeneous elements as above, these identities become
\[
\Delta(D(x))=\sum_{i} D(a_{i}) \otimes b_i \quad \quad \text{and} \quad \quad \Delta(D(x)) = \sum_{i} (-1)^{d|a_i|} a_i \otimes D(b_i)
\]
respectively. In particular, a graded coderivation is precisely the sum of a left and a right coderivation. Returning to our case, unlike the direct sum product, the restriction-contraction coproduct is asymmetric: the deletion-type differentials act on the contraction factor, while the contraction-type differentials act on the restriction factor. Thus the morphisms are only ever one-sided graded coderivations for $\Delta$. 

\begin{lemma}\label{lem:coderivation}
\begin{enumerate}
	\item On $\MC_{\bullet}$ both $\partial_{\del}$ and $\partial_{\clp}$ are right coderivations for $\Delta$, i.e., 
	\[
	\Delta \circ \partial_{\del} = (\id \otimes \partial_{\del}) \circ \Delta\quad \text{and} \quad 
	\Delta \circ \partial_{\clp}= (\id \otimes \partial_{\clp}) \circ \Delta.
	\]
	\item  On $\MC_{\bullet}$ both $\partial_{\con}$ and $\partial_{\lp}$ are left coderivations for $\Delta$, i.e., 
	\[
	\Delta \circ \partial_{\con}= (\partial_{\con} \otimes \id) \circ \Delta\quad \text{and} \quad 
	\Delta \circ \partial_{\lp} = (\partial_{\lp} \otimes  \id ) \circ \Delta.
	\]
	\end{enumerate}
\end{lemma}

The key point is conceptual rather than computational: in a term $[\sM|S,\eta|S]\otimes[\sM/S,\eta/S]$ of $\Delta([\sM,\eta])$, deleting an element $x\notin S$ changes only the contraction factor, while contracting an element $x\in S$ changes only the restriction factor. The only subtlety is the Koszul sign produced when the interior product with respect to $x$ is moved past the $|S|$ elements contributing to $\eta|S$.

\begin{proof}
Again we check the formulas on generators beginning with $\partial_{\del}$. For a class $[\sM,\eta]$ expanding the left-hand side gives
\begin{align*}
\Delta\left(\partial_{\del}([\sM,\eta])\right)
&= \sum_{\substack{x\in E(\sM)\\ \text{$x$ not a coloop}}} \Delta\left([\sM\setminus x,\eta\setminus x]\right) \\
&= \sum_{\substack{x\in E(\sM)\\ \text{$x$ not a coloop}}}\;\;\sum_{S\subseteq E(\sM)\setminus \{x\}}
[(\sM\setminus x)|S,(\eta\setminus x)|S]\otimes[(\sM\setminus x)/S,(\eta\setminus x)/S].
\end{align*}
For $x\not\in S$, deletion commutes with restriction and contraction by $S$, and so $(\sM\setminus x)|S = \sM|S$ and $(\sM\setminus x)/S = (\sM/S)\setminus x$. Further, for $x\not\in S$, $x$ is a coloop of $\sM/S$ if and only if $x$ is a coloop of $\sM$. Letting $\mu_{S}(\eta|S\otimes \eta/S)=\eta$, since $x\not\in S$:
\[
\eta\setminus x = \iota_{x}(\eta)=\iota_{x}\left(\eta|S\wedge \eta/S\right)=(-1)^{|S|}\eta|S\wedge \iota_{x}(\eta/S)=\mu_{S}\left(\eta|S\otimes (-1)^{|S|}(\eta/S)\setminus x\right).
\]
Thus, we can take $(\eta\setminus x)|S=\eta|S$ and $(\eta\setminus x)/S = (-1)^{|S|}(\eta/S)\setminus x$. Substituting these identities into our computation of $\Delta(\partial_{\del}([\sM,\eta]))$ and reindexing gives
\begin{align*}
\Delta\left(\partial_{\del}([\sM,\eta])\right)
&= \sum_{S\subseteq E(\sM)} \sum_{\substack{x\in E(\sM)\setminus S\\ \text{$x$ not a coloop in }\sM/S}} (-1)^{|S|}\left[\sM|S,\eta|S\right]\otimes\left[(\sM/S)\setminus x,(\eta/S)\setminus x\right] \\
&= \sum_{S\subseteq E(\sM)} (-1)^{|S|}\left[\sM|S,\eta|S\right]\otimes \partial_{\del}\left(\left[\sM/S,\eta/S\right]\right) = \left(\id\otimes \partial_{\del}\right)\left(\Delta\left([\sM,\eta]\right)\right) 
\end{align*}
The proofs in the remaining cases are analogous. 
% since $\partial_{\del}$ has degree $-1$ and $|[\sM|S,\eta|S]|=|S|$. The proof for $\partial_{\clp}$ is identical, replacing ``not a coloop'' by ``a coloop'' throughout.

%
\end{proof}

%%%%%%%%%%%%%%%%%%%%%%%%%%%%%%%%%%%%%%%%%%%%%%%%%%%
\subsection{Hopf Algebras \& Acyclicity}\label{subsec:dg-hopf-acyclic}
%%%%%%%%%%%%%%%%%%%%%%%%%%%%%%%%%%%%%%%%%%%%%%%%%%%

Having defined a product and coproduct on $\MC$ we now turn to how they can be combined. The most natural packaging would be a differential graded Hopf algebra, in which a single differential $\partial$ is simultaneously a graded derivation for $\star$ and a graded coderivation for $\Delta$. Lemma~\ref{lem:coderivation} shows this fails for each differential; this accounts for the asymmetry in Theorem~\ref{thm:hopf-algebra}, which we now prove.

\begin{proof}[Proof of Theorem~\ref{thm:hopf-algebra}]
By Lemma~\ref{lem:direct-sum-product} the direct sum product is associative and unital, while Lemma~\ref{lem:coproduct} shows that $\Delta$ is coassociative and counital. Lemma~\ref{lem:direct-sum-gradings}(2) and Lemma~\ref{lem:coproduct-gradings}(2) show that both structures are homogeneous of degree $0$ for the grading by ground set size. Thus to show that $(\MC_{\bullet},\star,[\varnothing,1],\Delta,\epsilon)$ is a connected graded Hopf algebra it remains only to check the usual compatibilities between $\star$ and $\Delta$.

First we check that $\Delta$ is an algebra map with respect to $\star$, i.e., $\Delta([\sM,\eta] \star [\sQ,\omega])=\Delta([\sM,\eta]) \star \Delta([\sQ,\omega])$. Here the product on the right-hand side is the extension of $\star$ to $\MC_{\bullet}\otimes \MC_{\bullet}$ by
\[
(a\otimes b)\star (c\otimes d)\coloneqq (-1)^{|b||c|}(a\star c)\otimes (b\star d).
\]
Towards this, expanding $\Delta([\sM,\eta] \star [\sQ,\omega])$ using the definitions gives
\[
\Delta\left(\left[\sM,\eta\right]\star\left[\sQ,\omega\right]\right) = \Delta\left(\left[ \sM\oplus \sQ, \eta\wedge \omega\right]\right) = \sum_{T\subseteq E(\sM \oplus\sQ)} \left[(\sM\oplus \sQ)|T,(\eta\wedge \omega)|T\right]\otimes\left[(\sM\oplus \sQ)/T,(\eta\wedge \omega)/T\right].
\]
Since $E(\sM\oplus\sQ)=E(\sM)\sqcup E(\sQ)$ every subset $T$ in the sum above decomposes uniquely as $T=S\sqcup U$ with $S\subseteq E(\sM)$ and $U\subseteq E(\sQ)$. For such a pair we have
\[
(\sM\oplus\sQ)|T = (\sM\oplus \sQ)|(S\sqcup U) = (\sM|S)\oplus(\sQ|U)
\qquad\text{and}\qquad
(\sM\oplus \sQ)/T = (\sM\oplus \sQ)/(S\sqcup U) = (\sM/S)\oplus(\sQ/U).
\]
Moreover, if $\mu_{S}(\eta|S\otimes \eta/S)=\eta$ and $\mu_{U}(\omega|U\otimes \omega/U)=\omega$, then
\[
\eta\wedge \omega = \eta|S\wedge \eta/S\wedge \omega|U\wedge \omega/U = (-1)^{|E(\sM)\setminus S|\,|U|} \left(\eta|S\wedge \omega|U\right)\wedge \left(\eta/S\wedge \omega/U\right).
\]
This means we may choose the orientations for the subsets $T=S\sqcup U$ so that
\begin{align*}
\left[(\sM\oplus \sQ)|T,(\eta\wedge \omega)|T\right]\otimes & \left[(\sM\oplus \sQ)/T,(\eta\wedge \omega)/T\right] \\
&= (-1)^{|E(\sM)\setminus S|\,|U|}
\left(\left[\sM|S,\eta|S\right]\star\left[\sQ|U,\omega|U\right]\right)\otimes
\left(\left[\sM/S,\eta/S\right]\star\left[\sQ/U,\omega/U\right]\right).
\end{align*}
Finally expressing $\Delta([\sM,\eta]\star[\sQ,\omega])$ as a sum over subsets of $E(\sM)$ and $E(\sQ)$  in this way gives
\begin{align*}
\Delta\left(\left[\sM,\eta\right]\star\left[\sQ,\omega\right]\right)
= \sum_{S,U} (-1)^{|E(\sM)\setminus S|\,|U|}
\left(\left[\sM|S,\eta|S\right]\star\left[\sQ|U,\omega|U\right]\right)\otimes
\left(\left[\sM/S,\eta/S\right]\star\left[\sQ/U,\omega/U\right]\right).
\end{align*}
 This is equal to $\Delta([\sM,\eta])\star \Delta([\sQ,\omega])$ as needed since $|E(\sQ|U)|= |U|$ and $|E(\sM/S)| = |E(\sM) \setminus S|$. 
 
To see that the counit is multiplicative, notice $\epsilon([\sM,\eta]\star[\sQ,\omega])$ is nonzero if and only if $\sM\oplus \sQ=\varnothing$, which happens precisely when both $\sM$ and $\sQ$ are empty. Similarly, from the definitions one sees that $\Delta([\varnothing,1])=[\varnothing,1]\otimes [\varnothing,1]$. Finally, $\MC_{\bullet}$ is connected since $\MC_{0}\cong \QQ\cdot [\varnothing,1]$.  This implies that $\MC_{\bullet}$ has an antipode by the standard connected graded bialgebra argument, finishing the proof that $(\MC_{\bullet},\star,[\varnothing,1],\Delta,\epsilon)$ is a connected graded Hopf algebra.

The statements in parts (1), (2), and (3) concerning the differentials are the content of Lemmas~\ref{lem:coderivation} and \ref{lem:direct-sum-derivation}  together with the fact that the sum of right (respectively left) coderivations is again a right (respectively  left) coderivation.  Finally, the last claim concerning the filtrations follows from Lemmas~\ref{lem:direct-sum-gradings} and \ref{lem:coproduct-gradings}, which show that the filtrations are compatible with $\star$ and $\Delta$ respectively. The descent to the associated graded complexes is then immediate. 
\end{proof}

\begin{corollary}\label{cor:dg-algebra}
	Both $(\MC_{\bullet},\star,[\varnothing,1], \partial_{\del}^{\tot})$ and $(\MC_{\bullet},\star,[\varnothing,1], \partial_{\con}^{\tot})$ are dg-algebras. 
\end{corollary}

\begin{proof}
This follows immediately from Theorem~\ref{thm:hopf-algebra}.
\end{proof}

\begin{corollary}\label{cor:structure-summary}
Let $\bP$ be a matroid property stable under isomorphism and direct sums.
\begin{enumerate}
    \item If $\bP$ is deletion \textup{(}resp.\ contraction\textup{)} closed, then $\MC^{\bP}_\bullet$ inherits the structure of a dg-algebra under $\star$, with multiplicative rank and nullity filtrations whose associated graded pieces are again dg-algebras.

    \item If $\bP$ is minor closed, then $\MC^{\bP}_\bullet$ is a filtered Hopf algebra as in Theorem~\ref{thm:hopf-algebra}, and the associated graded objects are Hopf algebras.

    \item Matroid duality exchanges the deletion-side and contraction-side structures, replacing $\bP$ by $\bP^*$.
\end{enumerate}
\end{corollary}

\begin{proof}
	For part (1), stability under direct sum makes $\star$ a well-defined product on $\MC_{\bullet}^{\bP}$ and Proposition~\ref{prop:property-totalizations} shows the relevant total differential and filtrations restrict. Hence the dg-algebra statements follow from Corollary~\ref{cor:dg-algebra}. For  part (2), if $\bP$ is minor closed then $\sM|S$ and $\sM/S$ also satisfy $\bP$ for all $S \subset E(\sM)$ and so $\Delta$ and $\epsilon$ restrict to $\MC_{\bullet}^{\bP}$. Theorem~\ref{thm:hopf-algebra} now applies verbatim. Part (3) is a direct application of duality. 
\end{proof}

We now turn to showing that the algebra structure on $\MC_{\bullet}$ implies that many variants of matroid complexes are acyclic. 

\begin{theorem}\label{thm:acyclicity-summary}
      	All of the following complexes are acyclic:
	\[
		\left(\MC_{\bullet}, \partial_{\del}\right), \quad \quad \left(\MC_{\bullet}, \partial_{\del}^{\tot} \right), 
		\quad \quad \left(\MC_{\bullet}, \partial_{\con} \right), \quad \text{and} \quad \left(\MC_{\bullet}, \partial_{\con}^{\tot} \right).
	\]
\end{theorem}

Before proving this, we recall the following basic lemma, certainly known to experts. 

\begin{lemma}\label{lem:dg-acyclic}
	Let $(C_{\bullet}, \partial, \star, 1)$ be a unital dg algebra. If there exists an element $\ell \in C_{1}$ such that $\partial(\ell) = 1$ then multiplication by $\ell$ gives a contracting homotopy; thus, $(C_{\bullet}, \partial)$ is acyclic.
\end{lemma}

\begin{proof}
	For an integer $n$, define $h_{n}:C_{n}\to C_{n+1}$ by $c \mapsto (-1)^{n}(c\star \ell)$. Using that $\partial$ is a graded derivation for $\star$ we have that 
    \begin{align*}
        \partial\left(h_{n}\left(c\right)\right) &= (-1)^{n} \partial\left( c \star \ell\right) = (-1)^{n}\partial\left(c\right)\star \ell + (-1)^{2n}c\star \partial(\ell) = (-1)^{n}\partial\left(c\right)\star \ell + c\star 1 
        \\
        &= (-1)^{n}\partial\left(c\right)\star \ell + c = -h_{n-1}\left(\partial\left(c\right)\right) + c
   \end{align*}
    Rearranging gives that $\partial\circ h_{n}+h_{n-1}\circ \partial$ is equal to the identity on $C_{n}$ as claimed. 
\end{proof}

The proof of Theorem~\ref{thm:acyclicity-summary} follows from the following corollary taking $\bP$ to be the trivial property of being a matroid. 

\begin{corollary}\label{cor:acyclicity-summary}
Let $\bP$ be a matroid property stable under isomorphism.
\begin{enumerate}
\item If $\bP$ is deletion closed and closed under direct sum with $U_{0,1}$, the following complexes are acyclic
\[
\left(\MC^{\bP}_\bullet,\partial^{\tot}_{\del}\right),\qquad
\left(\MC^{\bP}_\bullet,\partial_{\del}\right),\qquad \text{and} \qquad
\left(\MC^{\bP}_{\bullet,r},\partial_{\del}\right)\ \text{for all } r\ge 0.
\]

\item If $\bP$ is contraction closed and closed under direct sum with $U_{0,1}$, the following complexes are acyclic
\[
\left(\MC^{\bP}_\bullet,\partial^{\tot}_{\con}\right),\qquad
\left(\MC^{\bP}_\bullet,\partial_{\lp}\right),\qquad \text{and} \qquad
\left(\MC^{\bP}_{\bullet,r},\partial_{\lp}\right)\ \text{for all } r\ge 0.
\]

\item If $\bP$ is deletion closed and closed under direct sum with $U_{1,1}$, the following complexes are acyclic
\[
\left(\MC^{\bP}_\bullet,\partial^{\tot}_{\del}\right),\qquad
\left(\MC^{\bP}_\bullet,\partial_{\clp}\right),\qquad \text{and} \qquad
\left(\MC^{\bP}_{k,\bullet},\partial_{\clp}\right)\ \text{for all } k\ge 0.
\]

\item If $\bP$ is contraction closed and closed under direct sum with $U_{1,1}$, the following complexes are acyclic
\[
\left(\MC^{\bP}_\bullet,\partial^{\tot}_{\con}\right),\qquad
\left(\MC^{\bP}_\bullet,\partial_{\con}\right),\qquad \text{and} \qquad
\left(\MC^{\bP}_{k,\bullet},\partial_{\con}\right)\ \text{for all } k\ge 0.
\]
\end{enumerate}
\end{corollary}

The proof is driven by two elementary identities: adjoining a loop $U_{0,1}$ gives a degree-$1$ class whose boundary is the unit for the differentials $\partial_{\del}$, $\partial_{\del}^{\tot}$, $\partial_{\lp}$, and $\partial_{\con}^{\tot}$, while adjoining a coloop $U_{1,1}$ does the same for $\partial_{\clp}$, $\partial_{\del}^{\tot}$, $\partial_{\con}$, and $\partial_{\con}^{\tot}$. Every acyclicity statement in the corollary is then the same contracting-homotopy argument once one checks that the relevant direct sum preserves the stated subcomplex.

\begin{proof}
For part (1), let $E(U_{0,1})=\{\ell\}$ and fix an orientation $\eta$ on $U_{0,1}$. With appropriate choice of $\eta$ a computation shows that 
\[
\partial_{\del}\left([U_{0,1},\eta]\right)=\partial^{\tot}_{\del}\left([U_{0,1},\eta]\right)=[\varnothing,1]=1_{\star}
\quad \quad \text{and} \quad \quad 
\partial_{\lp}\left([U_{0,1},\eta]\right)=\partial^{\tot}_{\con}\left([U_{0,1},\eta]\right)=[\varnothing,1]=1_{\star}.
\]
Since $U_{0,1}$ has rank $0$ and nullity $1$, direct sum with $[U_{0,1},\eta]$ preserves the complexes
\[
(\MC_\bullet,\partial_{\del}^{\tot}),\qquad
(\MC_\bullet,\partial_{\del}),\qquad
(\MC_{\bullet,r},\partial_{\del}),\qquad
(\MC_\bullet,\partial_{\con}^{\tot}),\qquad
(\MC_\bullet,\partial_{\lp}),\qquad
(\MC_{\bullet,r},\partial_{\lp}),
\]
and, by the assumption on $\bP$, it also preserves the corresponding $\bP$-subcomplexes listed in the statement. Since the relevant differentials are graded derivations for $\star$ the same contracting-homotopy calculation as in Lemma~\ref{lem:dg-acyclic}, with $[U_{0,1},\eta]$, restricts to each of the above $\bP$-subcomplexes. This proves all the acyclicity claims involving $U_{0,1}$. The same argument applied to $[U_{1,1},\omega]$  with $\omega$ an appropriate choice of orientation gives the contracting homotopy in the remaining cases.
\end{proof}

\begin{corollary}
After tensoring with $\QQ$, the deletion and contraction complexes of \cite{ABGW00} have one-dimensional homology in degree $1$ and vanishing homology in all higher degrees.
\end{corollary}

\begin{proof}
For degrees $n\ge 2$, the positive-degree truncation agrees with the unreduced complex, so the homology vanishes by Theorem~\ref{thm:acyclicity-summary}. In degree $1$, the truncation has zero differential out of $C_1$, hence
\[
H_1 \cong C_1/\img(\partial_2).
\]
For the unreduced complex, acyclicity gives $\img(\partial_2)=\ker(\partial_1)$, and $H_0=0$ implies $\partial_1:C_1\to C_0\cong \QQ$ is surjective. Therefore
\[
H_1 \cong C_1/\ker(\partial_1)\cong C_0\cong \QQ.
\]
The explicit generators are obtained from the degree-$1$ calculation.
\end{proof}

\subsection{Connected Generators and Connected Quotients}
By Theorem~\ref{thm:hopf-algebra}, $\MC_{\bullet}$ is a connected graded Hopf algebra in the standard algebraic sense that its degree-zero piece is $\QQ\langle[\varnothing,1]\rangle$. It would be interesting to try to determine generators for $\MC_{\bullet}$ as a Hopf algebra. In this section we approach the slightly different question of generators of $\MC_{\bullet}$ as an algebra with respect to $\star$. Towards this, the matroids that are indecomposable with respect to direct sum should play the role of multiplicative atoms, just as connected graphs do in graph complexes built from disjoint union. The next theorem makes this precise.

A matroid is said to be \emph{connected} if and only if it cannot be written as the direct sum of two nonempty matroids. Note we do not regard the empty matroid as connected. Connectedness of matroids is preserved under both isomorphism and duality. Every nonempty matroid can be written as the direct sum of connected nonempty matroids; moreover, this decomposition is unique up to isomorphism and order. Given a positive integer $n$ let $V_{n} \subset \MC_{n}$ denote the span of classes $[\sM,\eta]$ such that $\sM$ is connected. Consider the following $\ZZ$-graded $\QQ$-vector spaces: 
\[
V_{\bullet} \coloneqq \bigoplus_{n\in\ZZ} V_{n}, \quad \quad V_{\mathrm{even}} \coloneqq \bigoplus_{k\in \ZZ} V_{2k}, \quad \text{and} \quad V_{\mathrm{odd}} \coloneqq \bigoplus_{k\in \ZZ} V_{2k+1}
\]
By definition there is an inclusion $V_{\bullet} \hookrightarrow \MC_{\bullet}$ of $\ZZ$-graded vector spaces. By Lemma~\ref{lem:direct-sum-gradings}, $(\MC_{\bullet},\star)$ is graded-commutative, which implies the inclusion $V_{\bullet} \hookrightarrow \MC_{\bullet}$ extends uniquely into a graded algebra morphism:
\[
	\begin{tikzcd}[row sep = .5em, column sep = 3.2em]
	\displaystyle \Sym\left(V_{\mathrm{even}}\right) \otimes \bigwedge V_{\mathrm{odd}} \arrow[r, "\Psi"] & \left(\MC_{\bullet}, \star\right) \\
	\end{tikzcd}
\]
Because the total grading controls the Koszul signs, even-degree connected classes should behave polynomially and odd-degree connected classes should behave exteriorly. Thus the real content of the next theorem is not merely that connected matroids generate $\MC_{\bullet}$---that already follows from the connected-component decomposition of matroids---but that the only multiplicative relations are the ones forced by graded-commutativity.

\begin{theorem}\label{thm:connected-generators}
	The morphism $\Psi$ is an isomorphism of graded algebras. Thus, $(\MC_{\bullet},\star)$ is the free super-commutative algebra on its connected classes.
\end{theorem}

\begin{proof}
	We explicitly construct an inverse to $\Psi$ on generators. Set $\Psi^{-1}([\varnothing,1])=1$. If $n>0$, given a class $[\sM,\eta] \in \MC_{n}$ let $\sM=\sM_{1}\oplus \sM_{2}\oplus \cdots \oplus \sM_{t}$ be a decomposition of $\sM$ into its connected components. We may choose orientations $\eta_{1},\ldots,\eta_{t}$ on the matroids $\sM_1,\ldots,\sM_{t}$ such that $\eta = \eta_{1} \wedge \cdots \wedge \eta_{t}$. Define $\Psi^{-1}([\sM,\eta])$ by
	\[
	\Psi^{-1}\left(\left[\sM,\eta\right]\right) \coloneqq \left[\sM_{1},\eta_{1} \right] \cdots \left[\sM_{t},\eta_{t}\right] \in \Sym\left(V_{\mathrm{even}}\right)\otimes\bigwedge V_{\mathrm{odd}}.
	\]
	Checking this is well-defined requires three steps: (i) The definition of $\Psi^{-1}$ is independent of the order of the connected components. This is true as changing the order of the connected components would introduce the same Koszul sign on both $\eta_{1}\wedge\cdots\wedge\eta_{t}$ and on $[\sM_{1},\eta_{1}]\cdots[\sM_{t},\eta_{t}]$. (ii) Choosing different orientations on the $\sM_{i}$ whose wedge is still equal to $\eta$ does not change $\Psi^{-1}$. This changes the product by signs that will cancel out. (iii) If $\phi:\sM\to \sM'$ is an isomorphism then $\Psi^{-1}([\sM,\eta]) = \Psi^{-1}([\sM',\phi_{*}(\eta)])$. One checks that $\phi$ will induce an isomorphism between the connected components compatible with orientations. 
	
	This is an algebra map with respect to the direct sum product $\star$ because the connected components of $\sM\oplus \sQ$ are precisely the connected components of $\sM$ together with the connected components of $\sQ$. Finally, one checks directly that $\Psi \circ \Psi^{-1}$ is the identity on $\MC_{\bullet}$ and since $\Psi^{-1} \circ \Psi$ is a graded algebra endomorphism of $\Sym(V_{\mathrm{even}})\otimes \bigwedge(V_{\mathrm{odd}})$,  which is the identity on the generators of $V_{\bullet}$ it is also the identity. Thus, $\Psi$ is an isomorphism.
\end{proof}

\begin{remark}
Theorem~\ref{thm:connected-generators} does not identify generators for $\MC_{\bullet}$ as a Hopf algebra. Indeed, if $\sM$ is connected then
\[
\Delta([\sM,\eta])=
[\varnothing,1]\otimes[\sM,\eta]
+[\sM,\eta]\otimes[\varnothing,1]
+\sum_{\varnothing\subsetneq S\subsetneq E(\sM)}
[\sM|S,\eta|S]\otimes[\sM/S,\eta/S],
\]
and so in general connected classes are typically far from primitive. Thus the theorem isolates the multiplicative structure of \(\MC_{\bullet}\), while the restriction-contraction coproduct remains genuinely nontrivial on connected generators.
\end{remark}

Theorem~\ref{thm:connected-generators} suggests passing from the full matroid complexes to a connected quotient. Since every disconnected matroid is a direct sum of smaller nonempty connected matroids, the disconnected part should be viewed as the decomposable part of the algebra. It is therefore natural to ask when the deletion and contraction differentials descend to the quotient by disconnected classes. This does not hold uniformly in all settings: Theorem~\ref{thm:acyclicity-summary} shows that in the presence of a loop or coloop the full complexes are often contractible for formal algebraic reasons. The next lemma shows that after excluding these one-element obstructions --- loopless on the deletion side and coloopless on the contraction side --- the disconnected classes do form a subcomplex.

If $\bP$ is a matroid property stable under isomorphism, let $\MC^{\bP,\disc}_{\bullet}$ denote the subspace of $\MC_{\bullet}^{\bP}$ spanned by classes $[\sM,\eta]$ where $\sM$ is a disconnected matroid. 

\begin{lemma}\label{lem:connected-sub}
\begin{enumerate}
	\item If $\bP$ is deletion closed and every matroid satisfying $\bP$ is loopless then $(\MC_{\bullet}^{\bP,\disc}, \partial_{\del})$ is a subcomplex of $(\MC_{\bullet}, \partial_{\del})$.
	\item If $\bP$ is contraction closed and every matroid satisfying $\bP$ is coloopless then $(\MC_{\bullet}^{\bP,\disc}, \partial_{\con})$ is a subcomplex of $(\MC_{\bullet}, \partial_{\con})$.
   \end{enumerate}
\end{lemma}

\begin{proof}
	Once again claim (2) follows from (1) by duality. It is enough to show that $\partial_{\del}(\MC_{n}^{\bP,\disc})$ is contained in $\MC_{n-1}^{\bP,\disc}$. This amounts to the claim that if $\sM$ is disconnected and  $x \in E(\sM)$ is not a coloop then $\sM\setminus x$ is also disconnected. To see this, suppose $\sM=\sM_{1}\oplus \sM_{2}$ with both $\sM_{1}$ and $\sM_{2}$ nonempty. Given $x\in E(\sM_1)$ we have that $\sM\setminus x = (\sM_{1}\setminus x)\oplus \sM_{2}$. Since every matroid with $\bP$ is loopless, the only one-element connected component that could disappear under deletion by $x$ is $U_{1,1}$. However, this would mean $\sM_1=U_{1,1}$ and $E(\sM_1)=\{x\}$ implying $x$ is a coloop. Hence $\sM_{1}\setminus x$ is nonempty and $\sM\setminus x$ remains the direct sum of two nonempty matroids so it is disconnected. 
\end{proof}

The upshot from Lemma~\ref{lem:connected-sub} is that we may quotient the ``full'' matroid complexes by the disconnected subcomplexes to form the connected deletion/contraction matroid complexes:
\[
\left( \MC_{\bullet}^{\bP,\cont}, \partial_{\del}\right) \quad \quad \text{and} \quad \quad \left( \MC_{\bullet}^{\bP,\cont}, \partial_{\con}\right)  \quad \quad \text{where} \quad \quad \MC_{\bullet}^{\bP,\cont} \coloneqq \MC_{\bullet}^{\bP}/\MC^{\bP,\disc}_{\bullet}.
\] 

Recall that if $A=A_{0}\oplus A_{+}$ is a graded algebra with product $\star$, then its indecomposable quotient is $A_{+}/(A_{+}\star A_{+})$, which measures the part of $A$ not generated by products of positive-degree elements. In the present setting the product is $\star$, so it is natural to compare the connected quotient of $\MC_{\bullet}^{\bP}$ with the indecomposable quotient of the graded algebra $(\MC_{\bullet}^{\bP},\star)$.

\begin{corollary}\label{cor:connected-indecomposables}
If $\bP$ is stable under isomorphism, closed under direct sums, and is inherited by connected components then $\MC^{\bP,\disc}_{\bullet} = \MC_{+}^{\bP} \star \MC_{+}^{\bP}$. Moreover, whenever Lemma~\ref{lem:connected-sub} applies, the connected quotient identifies with the indecomposable quotient
\[
\MC^{\bP,\cont}_{\bullet}
\cong
\MC^{\bP}_{+}/(\MC^{\bP}_{+}\star \MC^{\bP}_{+}).
\]
\end{corollary}

\begin{proof}
	The final statement follows immediately from the first. The first follows from the fact that every matroid can be written as a direct sum of connected components. 
\end{proof}

Corollary~\ref{cor:connected-indecomposables} identifies the connected quotient with the indecomposable quotient of the algebra \((\MC_{\bullet}^{\bP},\star)\). This is the matroid analogue of passing from all graphs to connected graphs in graph-complex theory. The point is not that the quotient inherits the full structure from Section~\ref{subsec:dg-hopf-acyclic}, but rather that it isolates the part of the theory not generated by repeated direct sums of smaller objects.

One should emphasize, however, that this is an algebraic quotient rather than a Hopf quotient. In general the subspace $\MC_{\bullet}^{\bP,\disc}$ need not be a co-ideal for $\Delta$, so there is no reason for $\MC_{\bullet}^{\bP,\cont}$ to inherit a compatible coproduct. Instead, Theorem~\ref{thm:connected-generators} shows that the product structure is already controlled by connected matroids, while Lemma~\ref{lem:connected-sub} shows that in favorable settings the homological study may likewise be reduced to connected classes. In this sense, the connected quotient is the natural reduced object to examine when searching for homology not forced by the free super-commutative algebra structure of the full complexes.

\section{Structural Results and Computations}\label{sec:computation}

The results of Section~\ref{sec:algebraic-structure} show that the direct sum product places surprisingly strong constraints on matroid complexes. The formal acyclicity statements of Theorem~\ref{thm:acyclicity-summary} explain why many unreduced complexes vanish, while Theorem~\ref{thm:connected-generators} and Corollary~\ref{cor:connected-indecomposables} show that connected classes are the natural place to look for genuinely new behavior. The goal of this section is therefore concrete: first isolate automatic vanishing phenomena, then choose explicit bases adapted to computation, and finally record the resulting calculations.

\subsection{Vanishing and Support Conditions}

\begin{lemma}\label{lem:orientation-reversing-zero}
Let $\sM$ be a matroid and $\eta$ an orientation on $\sM$. The class $[\sM,\eta]$ vanishes in $\MC_{\bullet}$ if and only if $\sM$ admits an orientation-reversing automorphism, that is, an automorphism $\tau:\sM\to \sM$ such that $\tau_{*}(\eta)=-\eta$.
\end{lemma}

\begin{proof}
If $\tau_{*}(\eta)=-\eta$, then in $\MC_{\bullet}$ we have
\[
[\sM,\eta]=[\sM,\tau_{*}(\eta)]=[\sM,-\eta]=- [\sM,\eta],
\]
so $2[\sM,\eta]=0$, and hence $[\sM,\eta]=0$ over $\QQ$. Conversely, first quotient $\widetilde{\cM}$ by the isomorphism relations and then by the orientation-reversal relations. After choosing one oriented representative for each isomorphism class, the only nontrivial way for a class to vanish is for some automorphism of $\sM$ to identify $\eta$ with $-\eta$. This is exactly the existence of an orientation-reversing automorphism.
\end{proof}

\begin{example}\label{exm:uniform-zero}
    Recall from Example~\ref{exm:into-uniform-matroid} the uniform matroid $U_{2,n}$ of rank $2$ on $n$ elements has $E=[n]$ and $\cB=\binom{[n]}{2}$. If $n\geq2$ then the transposition $\tau=(12)$ gives an isomorphism $\tau:U_{2,n}\to U_{2,n}$, which is orientation-reversing. Lemma~\ref{lem:orientation-reversing-zero} implies that $[U_{2,n},\eta]=0$ when $n\geq2$. 
\end{example}

As an immediate application of Lemma~\ref{lem:orientation-reversing-zero} we see that the matroid complex does not see matroids containing parallel or series elements or containing two or more loops or coloops. 

\begin{proposition}\label{prop:vanishing-2loops-parallel}
    Let $\sM$ be a matroid and $\eta$ an orientation. 
    \begin{enumerate}
        \item If $\sM$ has parallel elements then $[\sM,\eta]=0$.
        \item If $\sM$ has two or more loops then $[\sM,\eta]=0$.
        \item If $\sM$ has series elements then $[\sM,\eta]=0$.
        \item If $\sM$ has two or more coloops then $[\sM,\eta]=0$.
    \end{enumerate}
\end{proposition}

\begin{proof}
    Statements (3) and (4) concerning series elements and coloops follow from those in (1) and (2) since duality gives an isomorphism that takes parallel elements to series elements and loops to coloops. For (1) suppose that $\sM$ is a matroid and $x,y \in E(\sM)$ are parallel elements, so $\{x,y\}$ is a circuit. Let $\tau_{x,y}:E(\sM)\to E(\sM)$ be the transposition swapping $x$ and $y$ and fixing all other elements. We claim that $\tau_{x,y}$ is a matroid automorphism, which amounts to checking it takes bases to bases. Let $B \in \cB(\sM)$.  Since $\{x,y\}$ is a circuit, $B$ cannot contain both $x$ and $y$. If $B\cap\{x,y\}=\varnothing$, then $\tau_{x,y}(B)=B$. If $x\in B$ and $y\not \in B$, then $(B\setminus\{x\})\cup\{y\}$ has the same cardinality as $B$. If it were dependent, it would contain a circuit $C$ with $y\in C$. Applying circuit elimination to $C$ and the circuit $\{x,y\}$ at the element $y$ produces a circuit contained in $(C\setminus\{y\})\cup\{x\}\subseteq B$, contradicting that $B$ is a basis. Hence $(B\setminus\{x\})\cup\{y\}$ is independent, and therefore a basis. The case $y\in B$ and $x\notin B$ is symmetric. Thus $\tau_{x,y}$ preserves bases and is an automorphism of $\sM$. Since $\tau_{x,y}$ is a transposition, the induced map on $\bigwedge^{|E(\sM)|}\ZZ\langle E(\sM)\rangle$ is multiplication by $-1$. Therefore $(\tau_{x,y})_{*}(\eta)=-\eta$, and Lemma~\ref{lem:orientation-reversing-zero} implies that $[\sM,\eta]=0$.

For (2), suppose that $x,y\in E(\sM)$ are two loops. Then no basis of $\sM$ contains either $x$ or $y$, so the same transposition $\tau_{x,y}$ preserves every basis and hence is a matroid automorphism. Again $(\tau_{x,y})_{*}(\eta)=-\eta$, and Lemma~\ref{lem:orientation-reversing-zero} gives $[\sM,\eta]=0$.
\end{proof}

Note that since a matroid is simple if and only if it has no loops and no parallel elements, Proposition~\ref{prop:vanishing-2loops-parallel} implies that the simple and loopless subcomplexes are isomorphic. Similarly, it shows the cosimple and coloopless subcomplexes are isomorphic, and all rank-two matroids (regardless of ground set size) represent trivial classes in $\MC$. 

\begin{corollary}\label{cor-rank-2-trivial}
The rank-$2$ piece of the matroid complex is trivial, i.e., $\MC_{n-2,2}=0$ for all $n$.
\end{corollary}

\begin{proof}
Let $\sM$ be a rank-$2$ matroid. If $\sM$ is not simple, then either it has parallel elements or a loop. Parallel elements are handled by Proposition~\ref{prop:vanishing-2loops-parallel}. If $\sM$ has a loop, then either it has two loops, or deleting the unique loop leaves a loopless rank-$2$ matroid that is either uniform or still contains a dependent two-element set; in each case one obtains an orientation-reversing automorphism or parallel elements, so the class vanishes. If $\sM$ is simple, then every two-element subset of $E(\sM)$ is independent, hence every two-element subset is a basis. Therefore $\sM\cong U_{2,n}$, and Example~\ref{exm:uniform-zero} applies.
\end{proof}

Going forward we focus our attention primarily on the deletion complex as our computations are built around this case. With this in mind the next observation is the one most directly relevant for computation: a nonzero basis element that is a cycle in the deletion complex must already be supported on simple matroids.

\begin{lemma}\label{lem:ker-implies-loopless}
If $[\sM,\eta]\neq 0$ and $[\sM,\eta]\in \ker(\partial_{\del})$, then $\sM$ is simple.
\end{lemma}

\begin{proof}
By Proposition~\ref{prop:vanishing-2loops-parallel}, the condition $[\sM,\eta]\neq 0$ rules out parallel elements and shows that $\sM$ has at most one loop. Suppose, for contradiction, $\sM$ has a unique loop $\ell$. Then $\sM\setminus \ell$ is loopless, whereas every deletion $\sM\setminus x$ with $x\neq \ell$  has a loop. Hence the term $[\sM\setminus \ell,\eta\setminus \ell]$ in
\[
\partial_{\del}([\sM,\eta])=\sum_{\substack{x\in E(\sM)\\ x\text{ not a coloop}}}[\sM\setminus x,\eta\setminus x]
\]
cannot cancel with any other summand. Therefore $[\sM\setminus \ell,\eta\setminus \ell]=0$. But any orientation-reversing automorphism of $\sM\setminus \ell$ lifts across the extra loop $\ell$, so Lemma~\ref{lem:orientation-reversing-zero} would imply $[\sM,\eta]=0$, a contradiction.
\end{proof}

\subsection{Computing with Matroid Complexes}

Together, Lemma~\ref{lem:orientation-reversing-zero} and Lemma~\ref{lem:ker-implies-loopless} show that the computation is governed by two simple principles: generators with odd automorphisms vanish, and nonzero deletion cycles can only occur on simple matroids. This is the input for the basis construction below. Let $\omega_n = 1\wedge 2\wedge \cdots \wedge n$, viewed as an orientation on any matroid on the ground set $[n]$. For $i\in [n]$, let $\rho_i:[n]\setminus\{i\}\longrightarrow [n-1]$ be the unique order-preserving bijection. If $\sM$ is a matroid on $[n]$ and $\sM'$ is a matroid on $[n-1]$, define
\[
\mu(\sM\setminus i,\sM') \coloneqq
\begin{cases}
0 & \text{if $\sM\setminus i$ is not isomorphic to $\sM'$,}\\
\sgn(\psi\circ \rho_i^{-1}) & \text{if $\psi:\sM\setminus i\to \sM'$ is an isomorphism.}
\end{cases}
\]
Concretely, $\mu(\sM\setminus i, \sM')$ is the sign of the permutation needed to identify $\sM\setminus i$ with the chosen representative of $\sM'$ on the ground set $[n-1]$. Whenever $\Aut(\sM')\subset A_{n-1}$, this is well defined, since any two such isomorphisms differ by an even automorphism of $\sM'$.

The role of $\mu(\sM\setminus i,\sM')$ is to separate the two sources of signs in the deletion differential. The factor $(-1)^{i-1}$ comes from the interior product $\iota_i(\omega_n)$, while $\mu(\sM\setminus i,\sM')$ records the additional sign needed to relabel the deleted matroid back to the chosen representative on $[n-1]$. Proposition~\ref{prop:basis-mc} shows that, once these two signs are accounted for, the differential matrix is obtained simply by recording which deletions land in which isomorphism classes.

\begin{proposition}\label{prop:basis-mc}
Suppose $\widetilde{B}_n$ is a set of representatives for the isomorphism classes of matroids on $[n]$. The set $B_n \coloneqq \left\{(\sM,\omega_n)  \;\middle|\; \sM \in \widetilde{B}_n \;\; \Aut(\sM)\subset A_n\right\}$, with $\omega_n=1\wedge2\wedge\cdots\wedge n$, is a basis for $\MC_{n}$ and if $(\sM,\omega_n)\in B_n$:
    \[
    \partial_{\del}\left(\left[\sM,\omega_n\right]\right)
    =\sum_{\substack{i \in [n] \\  \text{$i$ not a coloop}}} \;\sum_{(\sM',\omega_{n-1})\in B_{n-1}}
    (-1)^{i-1}\,\mu(\sM\setminus i,\sM')\,[\sM',\omega_{n-1}].
    \]
\end{proposition}

\begin{proof}
Every class in degree $n$ has a representative of the form $(\sM,\omega_n)$ for some matroid on $[n]$. By Lemma~\ref{lem:orientation-reversing-zero}, such a class is nonzero precisely when $\sM$ has no odd automorphism, which proves the basis statement. The formula for the differential is obtained by deleting each $i\in [n]$, keeping track of the sign coming from interior product with respect to $i$, and then rewriting the result in the basis $B_{n-1}$.
\end{proof}

Similar results can be given for the other matroid complexes introduced in this paper; however, the above is the main input for the computations in the next section. 

\subsection{Implementation}
Proposition~\ref{prop:basis-mc} leads to a direct algorithm for constructing the deletion matroid complex from a list of matroid isomorphism classes. For each $n$, we enumerate matroid isomorphism classes on the ground set $[n]$, discard those with odd automorphisms, and use the surviving classes to form the basis $B_n$. For a basis element $[\sM,\omega_n]$, Proposition~\ref{prop:basis-mc} expresses $\partial_{\del}([\sM,\omega_n])$ as a signed sum of basis elements in $B_{n-1}$. In this way one obtains the differential matrices of $(\MC_{\bullet},\partial_{\del})$ bidegree by bidegree. We carried out two rounds of computation, each using different software and data sources. Most computations, including all of those for which we report runtime data, were run on an Ubuntu server with two 48-core 3.6 GHz AMD EPYC CPUs, 1.5 TB of RAM, and two GPUs: NVIDIA A100 80 GB and NVIDIA A40. 

All code and data discussed in this section are available via GitHub:
\begin{center}
\url{https://github.com/julietteBruce/matroidComplex}
\end{center}

\subsection{All Matroids ($n\leq 9$)}

We implemented the algorithm above in \textit{Macaulay2} \cite{M2}, using the \texttt{Matroids} package \cite{chenM2, chen19}, which provides basic matroid functions including deletion, contraction, isomorphism testing, and a database of all matroids, up to isomorphism, on at most nine elements. Some of the results can be seen in Figure~\ref{fig:data1}, which shows the graded pieces of $\MC$ graded by rank and ground set size. To compute combinatorial subcomplexes, we pruned the basis $B_n$ by testing whether the underlying matroid satisfied the desired property. For the simple and loopless subcomplexes we used the \textit{SageMath} database \texttt{matroids.AllMatroids}. For the binary, ternary, and regular subcomplexes we used \textit{OSCAR} and \textit{polymake}, since these properties can be checked efficiently via excluded-minor tests and related algorithms. The resulting lists of representatives were then imported back into \textit{Macaulay2} to assemble the corresponding differential matrices. The dimension tables for these subcomplexes appear in Appendix~\ref{sec:app}.

\begin{remark}
All entries in the above tables were computed in less than a minute, with the exceptions of the $(9,4)$ and $(9,5)$ entries. Computing each of these entries took roughly 6 hours of compute time. The bottleneck for these entries comes from the fact that there are $190{,}214$ matroids on 9 elements of rank 4 and by duality the same number of rank 5. Further, identifying whether a matroid on 9 elements has an odd automorphism becomes challenging as our code for computing matroid automorphisms is a bit naive and requires checking a large number of potential automorphisms. 
\end{remark}

\begin{example}
    The matroid complex $\left(\MC_{\bullet},\partial_{\del} \right)$ in degrees up to $n=7$ is given by
\[
\begin{tikzcd}[row sep = .4em, column sep = 3.3em]
%0 & 1 & 2 & 3 & 4 & 5 & 6 & 7 & \\
\QQ &\arrow[l, swap, "\partial_{\del,0}"] \QQ^{2} & \arrow[l,swap, "\partial_{\del,1}"] \QQ &\arrow[l,swap, "\partial_{\del,2}"] 0 &\arrow[l,swap, "\partial_{\del,3}"] 0 &\arrow[l,swap, "\partial_{\del,4}"] 0 &\arrow[l,swap, "\partial_{\del,5}"] \QQ^{2} & \arrow[l,swap, "\partial_{\del,6}"] \QQ^{18} & \arrow[l,swap, "\partial_{\del,7}"] \cdots
\end{tikzcd}.
\]
In our basis the two interesting nonzero maps $\partial_{\del,1}$ and $\partial_{\del,6}$ can be represented as the matrices below: 
\[
\partial_{\del,1}=\begin{bmatrix}0\\1\end{bmatrix},\quad  \partial_{\del,6}=\begin{bmatrix}1&3&0&4&4&0&3&1&0&0&0&0&0&0&0&0&0&0\\0&0&1&1&3&7&-1&0&1&0&0&0&0&0&0&0&0&0\end{bmatrix}.
\]
\end{example}

\subsection{Simple, Connected, Loopless Regular Matroids ($n\leq 15$)}

Since the number of matroids on 10 or more elements is unknown, extending our computations for the full matroid complex -- and most of the combinatorial subcomplexes -- beyond $n=9$ is presently out of reach. However, the structural results of Section~\ref{sec:algebraic-structure} suggest that the most interesting part of the complex can be isolated; by Corollary~\ref{cor:connected-indecomposables} the connected quotient of any loopless subcomplex identifies with the indecomposable quotient of the corresponding algebra, so connected matroids capture precisely the part of the complex not generated by direct sums of smaller objects. In particular, these complexes are most likely to exhibit interesting homology classes.
 
For regular matroids a much larger computation is feasible. The database of Fripertinger and Wild \cite{fripertingerWild} provides a complete list of isomorphism-class representatives for all simple loopless connected (slc) regular matroids through ground-set size $n=15$. These matroids are given via their $\FF_2$-representations.

The key computational bottleneck in computing matroid complexes is determining whether a matroid has an odd automorphism. When $\sM$ is binary, its automorphism group acts on the column space of the representing matrix over $\FF_2$, and this action can be analyzed efficiently using linear algebra over $\FF_2$. For each matroid in the Fripertinger--Wild database we computed the automorphism group by working directly with the $\FF_2$-representation. This approach is substantially more efficient than the naive method used in our full computations, which was the primary bottleneck in the earlier computation.

\begin{figure}[ht]
\centering
\includegraphics[scale=.75]{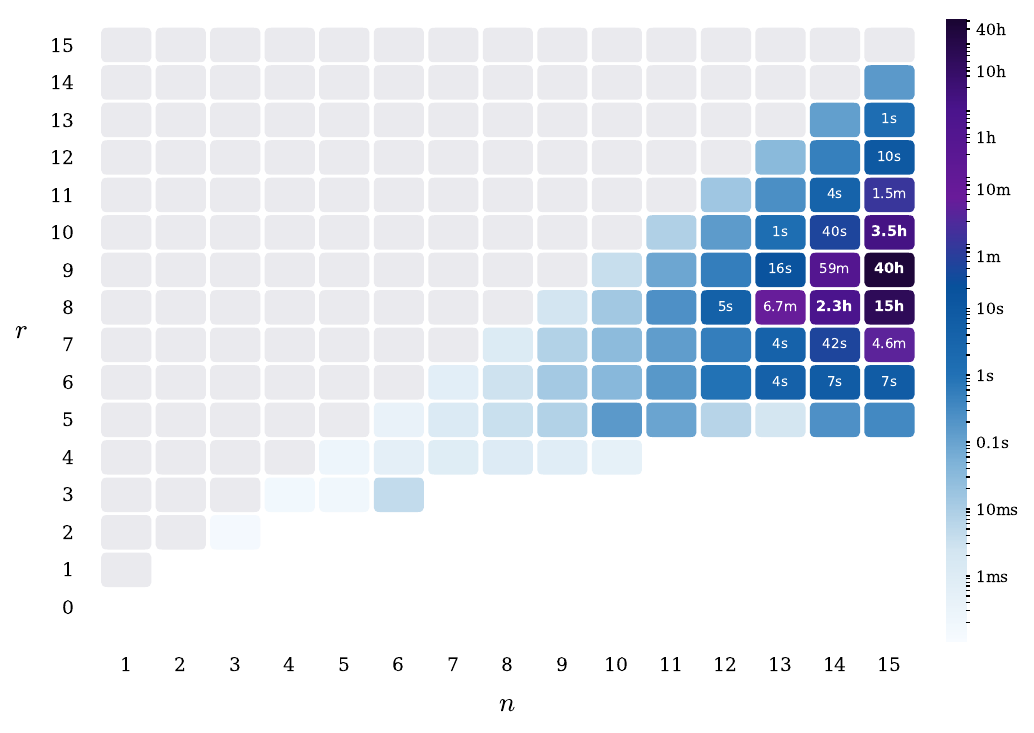}
\caption{CPU time for computing the basis of $\MC^{\reg,\simp,\cont}_{n-r,r}$ at each bidegree $(n,r)$. Empty cells indicate bidegrees with no connected simple loopless regular matroids. The total CPU time was approximately $70$ hours.}
\label{fig:timing}
\end{figure}
 
Using this method, we computed the connected quotient complex
\[
\left(\MC_{\bullet}^{\reg, \simp, \cont}, \partial_{\del}\right)
\]
on simple, loopless, connected regular matroids through ground-set size $n=15$. By Corollary~\ref{cor:connected-indecomposables}, this is the indecomposable quotient of the simple loopless regular matroid deletion subcomplex $(\MC_{\bullet}^{\reg,\simp},\partial_{\del})$. The dimensions of the chain groups in this complex are recorded in Figure~\ref{fig:reg-connected}. Note the symmetry of the table arrises because while the class of simple matroids is not duality closed, the condition of having no odd automorphism forces cosimplicity when $n>1$, which in turn makes the class of simple, loopless, connected regular matroids with no odd automorphism stable under duality. 

\begin{remark}\label{rmk:timing}
The computational profile for constructing the basis elements for this chain complex is recorded in Figure~\ref{fig:timing}. Through $n=12$ the entire computation takes only seconds, but the cost grows rapidly: the $n=14$ row requires roughly $3.3$ hours and $n=15$ requires roughly $60$ hours, for a total of approximately $70$ CPU hours. The bottleneck is concentrated at the middle ranks: the single entry $(n,r)=(15,9)$ accounts for over $40$ hours, reflecting the fact that the Fripertinger--Wild database contains $18,469$ connected simple loopless regular matroids at this bidegree --- the largest in the census. Despite this, the computation remains feasible because the $\FF_2$ linear algebra approach to detecting odd automorphisms is far more efficient than the naive method used in the first round of computations.
\end{remark}

\subsection{Computing Homology}

For the binary, ternary, and regular deletion matroid complexes, the acyclicity results of Section~\ref{subsec:dg-hopf-acyclic} already imply that the homology vanishes, so these computations serve primarily as a consistency check on the implementation. There are still potentially interesting combinatorial properties of these complexes; see \cite{BBC23}, where a virtual analogue of the Euler characteristic of the binary matroid complex is defined and computed.

On the other hand, the homology of the simple and loopless deletion matroid complexes (which are isomorphic by Proposition~\ref{prop:vanishing-2loops-parallel}) is more interesting, since Theorem~\ref{thm:acyclicity-summary} does not apply and we see no reason to expect them to be acyclic. We computed the homology of these subcomplexes through ground-set size $8$ and found they have two nonzero homology groups in this range:
\[
H_{i}\left(\MC_{\bullet}^{\lpl}, \partial_{\del}\right) = H_{i}\left(\MC_{\bullet}^{\simp}, \partial_{\del}\right) = 
\begin{cases}
 \QQ & \text{ if $i=0,1$,} \\
 0 & \text{ for $1 < i \leq 8$}.
\end{cases}
\]
These homology classes are easy to understand. In degree zero the class is generated by $[\varnothing,1]$. The matroid $U_{1,1}$ is simple, loopless, and has no odd automorphism, meaning it represents a nonzero class in  $\MC_{1}^{\lpl} \cong \MC_{1}^{\simp}$. Further, the unique element of $U_{1,1}$ is a coloop, implying that $\partial_{\del}([U_{1,1},\eta])=0$. On the other hand, $\MC_{2}^{\lpl}=\MC_{2}^{\simp}=0$ since there are only four matroids on two elements up to isomorphism: $U_{0,2}$ and $U_{0,1}\oplus U_{1,1}$, which are not loopless (and thus not simple), $U_{1,2}$ which has a parallel element, and $U_{2,2}$ which has rank 2. The classes of these last two matroids are zero by Proposition~\ref{prop:vanishing-2loops-parallel} and Corollary~\ref{cor-rank-2-trivial}. Hence $U_{1,1}$ represents a nontrivial homology class in both subcomplexes. 

The computation of the connected quotient complex $(\MC_{\bullet}^{\reg,\simp,\cont},\partial_{\del})$ through $n=15$ provides a substantial extension of the data in the regular case. By Corollary~\ref{cor:connected-indecomposables}, this connected quotient is precisely the indecomposable quotient of $(\MC_{\bullet}^{\reg,\simp},\partial_{\del})$, so its homology detects all classes not generated by direct sums. Since the property of being regular is minor closed and closed under direct sum with both $U_{0,1}$ and $U_{1,1}$, Theorem~\ref{thm:acyclicity-summary} implies that the full regular deletion complex is acyclic. However, the simple loopless restriction breaks the hypotheses of Theorem~\ref{thm:acyclicity-summary}.

Our computations find four nonzero homology groups of $\MC^{\reg,\simp,\cont}_{\bullet,\bullet}$ for $n\leq 14$ living in degrees:
\[
(n,r)=(1,1), \quad (6,3), \quad (10,5), \quad (14,7).
\]
Each of the corresponding homology groups is one-dimensional. Motivated by this pattern, we identified the matroids whose classes generate these homology groups, in order:
\[
U_{1,1}, \quad \sM(W_{3}), \quad \sM(W_{5}), \quad \sM(W_{7})
\]
where $\sM(G)$ is the graphic matroid of a graph $G$ and $W_{g}$ is the wheel graph of genus $g$ with $2g$ edges. This leads to the following conjecture:

\begin{conjecture}\label{con:reg-homology}
    The homology of the indecomposable quotient of the simple loopless regular matroid complex is generated by the classes of $U_{1,1}$ and $\sM(W_{2k+1})$ for all $k\geq1$. 
\end{conjecture}

Note a direct computation to show that $[\sM(W_{2k+1}),\eta]$ is nonzero in $\MC_{\bullet}$ for any choice of orientation $\eta$. Deletion on $\sM(W_{2k+1})$ corresponds to deleting an edge of $W_{2k+1}$. The graph $W_{2k+1}\setminus e$ has an odd automorphism for any edge $e$, and so $[\sM(W_{2k+1}),\eta]$ is a cycle. However, showing that this is not a boundary seems more difficult, as does showing it is the unique nonzero homology class for $n>1$.

The appearance of the odd wheels is especially suggestive in light of the
graph-complex picture. The classes above, apart from the one-element class
\(U_{1,1}\), lie in the graphic locus of the regular matroid complex. Brown--Chan--Galatius--Payne show that the odd wheel graphs $W_{2k+1}$ give nonzero classes in the degree-zero homology of the graph complex $\GC_{2}^{\vee}$, by showing they pair nontrivially with canonical forms $\omega_{4k+1}$. Further, the odd wheel classes generate a polynomial subalgebra of the $E^1$-page of the homological Quillen spectral sequence \cite{BCGP24}. Conjecture~\ref{con:reg-homology} may therefore be viewed as a matroidal shadow of the odd-wheel phenomenon in graph homology. 
%%%%%%%%%%%%%%%%%%%%%%%%%%%%%%%%%%%%%%%%%%%%%%%%%%%
%%%%%%%%%%%%%%%%%%%%%%%%%%%%%%%%%%%%%%%%%%%%%%%%%%%

%%%%%%%%%%%%%%%%%%%%%%%%%%%%%%%%%%%%%%%%%%%%%%%%%%%
%%%%%%%%%%%%%%%%%%%%%%%%%%%%%%%%%%%%%%%%%%%%%%%%%%%

\bibliographystyle{alphaurl}
\bibliography{main}

\appendix 
\section{Computed Complexes}
\label{sec:app}
The tables in this section record dimensions of chain groups for the regular, binary, ternary, simple, and loopless deletion matroid complexes, together with the simple loopless connected regular matroid complex.

\begin{figure}[ht]
\centering
\scalebox{0.8}{
\matroidtable{%
$9$ & \E & \E & \E & \E & \E & \E & \E & \E & \E & $0$ \\[2pt]
$8$ & \E & \E & \E & \E & \E & \E & \E & \E & $0$ & $0$ \\[2pt]
$7$ & \E & \E & \E & \E & \E & \E & \E & $0$ & $0$ & $0$ \\[2pt]
$6$ & \E & \E & \E & \E & \E & \E & $0$ & $0$ & $0$ & $0$ \\[2pt]
$5$ & \E & \E & \E & \E & \E & $0$ & $0$ & $0$ & $0$ & $0$ \\[2pt]
$4$ & \E & \E & \E & \E & $0$ & $0$ & $0$ & $\QQ$ & $\QQ$ & $0$ \\[2pt]
$3$ & \E & \E & \E & $0$ & $0$ & $0$ & $\QQ$ & $\QQ$ & $0$ & $0$ \\[2pt]
$2$ & \E & \E & $0$ & $0$ & $0$ & $0$ & $0$ & $0$ & $0$ & $0$ \\[2pt]
$1$ & \E & $\QQ$ & $\QQ$ & $0$ & $0$ & $0$ & $0$ & $0$ & $0$ & $0$ \\[2pt]
$0$ & $\QQ$ & $\QQ$ & $0$ & $0$ & $0$ & $0$ & $0$ & $0$ & $0$ & $0$ \\[2pt]
\midrule
 & \bfseries$0$ & \bfseries$1$ & \bfseries$2$ & \bfseries$3$ & \bfseries$4$ & \bfseries$5$ & \bfseries$6$ & \bfseries$7$ & \bfseries$8$ & \bfseries$9$ \\
}}
\caption{Dimensions of the chain groups for the regular matroid subcomplex $\MC^{\reg}_{n-r,r}$ for $0\le r\le n\le 9$. Blue entries indicate bidegrees with no matroids.}
\label{fig:data}
\end{figure}
 
%% ── Binary ───────────────────────────────────────────────────────────────
\begin{figure}[ht]
\centering
\scalebox{0.8}{
\matroidtable{%
$9$ & \E & \E & \E & \E & \E & \E & \E & \E & \E & $0$ \\[2pt]
$8$ & \E & \E & \E & \E & \E & \E & \E & \E & $0$ & $0$ \\[2pt]
$7$ & \E & \E & \E & \E & \E & \E & \E & $0$ & $0$ & $0$ \\[2pt]
$6$ & \E & \E & \E & \E & \E & \E & $0$ & $0$ & $0$ & $0$ \\[2pt]
$5$ & \E & \E & \E & \E & \E & $0$ & $0$ & $0$ & $\QQ$ & $\QQ^{4}$ \\[2pt]
$4$ & \E & \E & \E & \E & $0$ & $0$ & $0$ & $\QQ^{2}$ & $\QQ^{5}$ & $\QQ^{4}$ \\[2pt]
$3$ & \E & \E & \E & $0$ & $0$ & $0$ & $\QQ$ & $\QQ^{2}$ & $\QQ$ & $0$ \\[2pt]
$2$ & \E & \E & $0$ & $0$ & $0$ & $0$ & $0$ & $0$ & $0$ & $0$ \\[2pt]
$1$ & \E & $\QQ$ & $\QQ$ & $0$ & $0$ & $0$ & $0$ & $0$ & $0$ & $0$ \\[2pt]
$0$ & $\QQ$ & $\QQ$ & $0$ & $0$ & $0$ & $0$ & $0$ & $0$ & $0$ & $0$ \\[2pt]
\midrule
 & \bfseries$0$ & \bfseries$1$ & \bfseries$2$ & \bfseries$3$ & \bfseries$4$ & \bfseries$5$ & \bfseries$6$ & \bfseries$7$ & \bfseries$8$ & \bfseries$9$ \\
}}
\caption{Dimensions of the chain groups for the binary matroid subcomplex $\MC^{\FF_{2}}_{n-r,r}$ for $0\le r\le n\le 9$. Blue entries indicate bidegrees with no matroids.}
\label{fig:data2}
\end{figure}
 
%% ── Ternary ──────────────────────────────────────────────────────────────
\begin{figure}[ht]
\centering
\scalebox{0.8}{
\matroidtable{%
$9$ & \E & \E & \E & \E & \E & \E & \E & \E & \E & $0$ \\[2pt]
$8$ & \E & \E & \E & \E & \E & \E & \E & \E & $0$ & $0$ \\[2pt]
$7$ & \E & \E & \E & \E & \E & \E & \E & $0$ & $0$ & $0$ \\[2pt]
$6$ & \E & \E & \E & \E & \E & \E & $0$ & $0$ & $0$ & $0$ \\[2pt]
$5$ & \E & \E & \E & \E & \E & $0$ & $0$ & $0$ & $\QQ^{2}$ & $\QQ^{37}$ \\[2pt]
$4$ & \E & \E & \E & \E & $0$ & $0$ & $0$ & $\QQ^{4}$ & $\QQ^{15}$ & $\QQ^{37}$ \\[2pt]
$3$ & \E & \E & \E & $0$ & $0$ & $0$ & $\QQ^{2}$ & $\QQ^{4}$ & $\QQ^{2}$ & $0$ \\[2pt]
$2$ & \E & \E & $0$ & $0$ & $0$ & $0$ & $0$ & $0$ & $0$ & $0$ \\[2pt]
$1$ & \E & $\QQ$ & $\QQ$ & $0$ & $0$ & $0$ & $0$ & $0$ & $0$ & $0$ \\[2pt]
$0$ & $\QQ$ & $\QQ$ & $0$ & $0$ & $0$ & $0$ & $0$ & $0$ & $0$ & $0$ \\[2pt]
\midrule
 & \bfseries$0$ & \bfseries$1$ & \bfseries$2$ & \bfseries$3$ & \bfseries$4$ & \bfseries$5$ & \bfseries$6$ & \bfseries$7$ & \bfseries$8$ & \bfseries$9$ \\
}}
\caption{Dimensions of the chain groups for the ternary matroid subcomplex $\MC^{\FF_{3}}_{n-r,r}$ for $0\le r\le n\le 9$. Blue entries indicate bidegrees with no matroids.}
\label{fig:data3}
\end{figure}
 
%% ── Simple ───────────────────────────────────────────────────────────────
\begin{figure}[ht]
\centering
\scalebox{0.8}{
\matroidtable{%
$9$ & \E & \E & \E & \E & \E & \E & \E & \E & \E & $0$ \\[2pt]
$8$ & \E & \E & \E & \E & \E & \E & \E & \E & $0$ & $0$ \\[2pt]
$7$ & \E & \E & \E & \E & \E & \E & \E & $0$ & $0$ & $0$ \\[2pt]
$6$ & \E & \E & \E & \E & \E & \E & $0$ & $0$ & $0$ & $\QQ^{131}$ \\[2pt]
$5$ & \E & \E & \E & \E & \E & $0$ & $0$ & $0$ & $\QQ^{21}$ & $\QQ^{174{,}771}$ \\[2pt]
$4$ & \E & \E & \E & \E & $0$ & $0$ & $0$ & $\QQ^{9}$ & $\QQ^{290}$ & $\QQ^{174{,}502}$ \\[2pt]
$3$ & \E & \E & \E & $0$ & $0$ & $0$ & $\QQ^{2}$ & $\QQ^{7}$ & $\QQ^{14}$ & $\QQ^{117}$ \\[2pt]
$2$ & \E & \E & $0$ & $0$ & $0$ & $0$ & $0$ & $0$ & $0$ & $0$ \\[2pt]
$1$ & \E & $\QQ$ & $0$ & $0$ & $0$ & $0$ & $0$ & $0$ & $0$ & $0$ \\[2pt]
$0$ & $\QQ$ & $0$ & $0$ & $0$ & $0$ & $0$ & $0$ & $0$ & $0$ & $0$ \\[2pt]
\midrule
 & \bfseries$0$ & \bfseries$1$ & \bfseries$2$ & \bfseries$3$ & \bfseries$4$ & \bfseries$5$ & \bfseries$6$ & \bfseries$7$ & \bfseries$8$ & \bfseries$9$ \\
}}
\caption{Dimensions of the chain groups for the simple matroid and loopless matroid subcomplexes $\MC^{\simp}_{n-r,r}$ and $\MC^{\lpl}_{n-r,r}$ for $0\le r\le n\le 9$. Note these complexes are isomorphic by Proposition~\ref{prop:vanishing-2loops-parallel}. Blue entries indicate bidegrees with no matroids.}
\label{fig:simple}
\end{figure}

\begin{figure}[ht]
\centering
\scalebox{0.7}{
\begin{tikzpicture}
\node[inner sep=0pt] (tbl) {%
\renewcommand{\arraystretch}{1.5}%
\setlength{\tabcolsep}{4pt}%
\begin{tabular}{r | *{16}{>{\centering\arraybackslash}m{2.0em}}}
%% r=15
$15$ & \E & \E & \E & \E & \E & \E & \E & \E & \E & \E & \E & \E & \E & \E & \E & $0$ \\[2pt]
%% r=14
$14$ & \E & \E & \E & \E & \E & \E & \E & \E & \E & \E & \E & \E & \E & \E & $0$ & $0$ \\[2pt]
%% r=13
$13$ & \E & \E & \E & \E & \E & \E & \E & \E & \E & \E & \E & \E & \E & $0$ & $0$ & $0$ \\[2pt]
%% r=12
$12$ & \E & \E & \E & \E & \E & \E & \E & \E & \E & \E & \E & \E & $0$ & $0$ & $0$ & $0$ \\[2pt]
%% r=11
$11$ & \E & \E & \E & \E & \E & \E & \E & \E & \E & \E & \E & $0$ & $0$ & $0$ & $0$ & $0$ \\[2pt]
%% r=10
$10$ & \E & \E & \E & \E & \E & \E & \E & \E & \E & \E & $0$ & $0$ & $0$ & $0$ & $0$ & $\QQ$ \\[2pt]
%% r=9
$9$ & \E & \E & \E & \E & \E & \E & \E & \E & \E & $0$ & $0$ & $0$ & $0$ & $0$ & $\QQ$ & $\QQ^{18}$ \\[2pt]
%% r=8
$8$ & \E & \E & \E & \E & \E & \E & \E & \E & $0$ & $0$ & $0$ & $0$ & $0$ & $0$ & $\QQ^{21}$ & $\QQ^{247}$ \\[2pt]
%% r=7
$7$ & \E & \E & \E & \E & \E & \E & \E & $0$ & $0$ & $0$ & $0$ & $0$ & $0$ & $\QQ^{12}$ & $\QQ^{98}$ & $\QQ^{247}$ \\[2pt]
%% r=6
$6$ & \E & \E & \E & \E & \E & \E & $0$ & $0$ & $0$ & $0$ & $0$ & $\QQ$ & $\QQ^{5}$ & $\QQ^{12}$ & $\QQ^{21}$ & $\QQ^{18}$ \\[2pt]
%% r=5
$5$ & \E & \E & \E & \E & \E & $0$ & $0$ & $0$ & $0$ & $0$ & $\QQ^{2}$ & $\QQ$ & $0$ & $0$ & $\QQ$ & $\QQ$ \\[2pt]
%% r=4
$4$ & \E & \E & \E & \E & $0$ & $0$ & $0$ & $0$ & $0$ & $0$ & $0$ & $0$ & $0$ & $0$ & $0$ & $0$ \\[2pt]
%% r=3
$3$ & \E & \E & \E & $0$ & $0$ & $0$ & $\QQ$ & $0$ & $0$ & $0$ & $0$ & $0$ & $0$ & $0$ & $0$ & $0$ \\[2pt]
%% r=2
$2$ & \E & \E & $0$ & $0$ & $0$ & $0$ & $0$ & $0$ & $0$ & $0$ & $0$ & $0$ & $0$ & $0$ & $0$ & $0$ \\[2pt]
%% r=1
$1$ & \E & $\QQ$ & $0$ & $0$ & $0$ & $0$ & $0$ & $0$ & $0$ & $0$ & $0$ & $0$ & $0$ & $0$ & $0$ & $0$ \\[2pt]
%% r=0
$0$ & $0$ & $0$ & $0$ & $0$ & $0$ & $0$ & $0$ & $0$ & $0$ & $0$ & $0$ & $0$ & $0$ & $0$ & $0$ & $0$ \\[2pt]
\midrule
 & \bfseries$0$ & \bfseries$1$ & \bfseries$2$ & \bfseries$3$ & \bfseries$4$ & \bfseries$5$ & \bfseries$6$ & \bfseries$7$ & \bfseries$8$ & \bfseries$9$ & \bfseries$10$ & \bfseries$11$ & \bfseries$12$ & \bfseries$13$ & \bfseries$14$ & \bfseries$15$ \\
\end{tabular}%
};
\node[left=10pt of tbl.west, rotate=90, anchor=center, font=\itshape] {$r$};
\node[below=6pt of tbl.south, anchor=center, font=\itshape] {$n$};
\end{tikzpicture}%
}
\caption{Dimensions of the chain groups for the connected quotient of the simple loopless regular deletion matroid subcomplex $\MC^{\reg,\simp,\cont}_{n-r,r}$ for $0\le r\le n\le 15$. By Corollary~\ref{cor:connected-indecomposables}, this is the indecomposable quotient of $(\MC_{\bullet}^{\reg,\simp},\partial_{\del})$. Blue entries indicate bidegrees with no matroids.}
\label{fig:reg-connected}
\end{figure}
%%%%%%%%%%%%%%%%%%%%%%%%%%%%%%%%%%%%%%%%%%%%%%%%%%%
%%%%%%%%%%%%%%%%%%%%%%%%%%%%%%%%%%%%%%%%%%%%%%%%%%%
\clearpage
\end{document}